\documentclass[11pt]{article}

\usepackage{girth}

\newif\ifnotes
\notesfalse

\newcommand{\Nbr}{\mathrm{Nbr}}
\newcommand{\LNbr}{\mathrm{LNbr}}
\newcommand{\RNbr}{\mathrm{RNbr}}
\newcommand{\UN}{\mathrm{UN}}

\newcommand{\Blue}{\mathrm{Blue}}
\newcommand{\Red}{\mathrm{Red}}
\newcommand{\Cay}{\mathrm{Cay}}
\newcommand{\Inc}{\mathrm{Inc}}
\newcommand{\Skel}{\mathrm{Skel}}
\newcommand{\Usat}{U_{\mathrm{sat}}}
\newcommand{\Uunsat}{U_{\mathrm{unsat}}}
\newcommand{\TruncCay}{\mathrm{TruncCay}}

\begin{document}
\title{Explicit Two-Sided Vertex Expanders Beyond the Spectral Barrier}

\author{Jun-Ting Hsieh\thanks{Carnegie Mellon University. \texttt{juntingh@cs.cmu.edu}. Supported by NSF CAREER Award \#2047933.}
\and Ting-Chun Lin\thanks{UCSD. Hon Hai Research Institute. \texttt{til022@ucsd.edu}. Supported in part by DOE under the cooperative research agreement DE-SC0009919 and by the Simons Collaboration on Ultra-Quantum Matter (652264, JM).}
\and Sidhanth Mohanty\thanks{MIT. \texttt{sidm@mit.edu}. Supported by NSF Award DMS-2022448.}
\and Ryan O'Donnell\thanks{Carnegie Mellon University. \texttt{odonnell@cs.cmu.edu}. Supported in part by a grant from Google Quantum AI.}
\and Rachel Yun Zhang\thanks{MIT.  \texttt{rachelyz@mit.edu}. Supported by NSF Graduate Research Fellowship 2141064. Supported in part by DARPA under Agreement No. HR00112020023 and by an NSF grant CNS-2154149.}}

%% TITLE AND ABSTRACT
\sloppy
\maketitle
\begin{abstract}
We construct the first explicit two-sided vertex expanders that bypass the spectral barrier.

Previously, the strongest known explicit vertex expanders were given by $d$-regular Ramanujan graphs, whose spectral properties imply that every small subset of vertices $S$ has at least $0.5d|S|$ distinct neighbors.
However, it is possible to construct Ramanujan graphs containing a small set $S$ with no more than $0.5d|S|$ neighbors.
In fact, no explicit construction was known to break the $0.5 d$-barrier.

In this work, we give an explicit construction of an infinite family of $d$-regular graphs (for large enough $d$) where every small set expands by a factor of $\approx 0.6d$.
More generally, for large enough $d_1,d_2$, we give an infinite family of $(d_1,d_2)$-biregular graphs where small sets on the left expand by a factor of $\approx 0.6d_1$, and small sets on the right expand by a factor of $\approx 0.6d_2$.
In fact, our construction satisfies an even stronger property: small sets on the left and right have \emph{unique-neighbor expansion} $0.6d_1$ and $0.6d_2$ respectively.

Our construction follows the \emph{tripartite line product} framework of~\cite{HMMP24}, and instantiates it using the face-vertex incidence of the $4$-dimensional Ramanujan clique complex as its base component. As a key part of our analysis, we derive new bounds on the \emph{triangle density} of small sets in the Ramanujan clique complex.
\end{abstract}
\thispagestyle{empty}
\setcounter{page}{0}

\newpage

% TABLE OF CONTENTS
\enlargethispage{1cm}
\tableofcontents
\pagenumbering{roman}
\setcounter{page}{0}

\newpage
\pagenumbering{arabic}

\section{Introduction}

In this work, we study the problem of constructing explicit \emph{vertex expanders}.
Vertex expansion refers to the property that every ``small enough'' set of vertices should have ``many'' distinct neighbors.
Henceforth, we restrict our attention to bipartite graphs.
For $d_L > d_R$, we say that a $(d_L, d_R)$-biregular graph $G$ on $(L,R)$ is a \emph{$\gamma$-one-sided vertex expander} if every subset $S\subseteq L$ of size at most $\eta |L|$ has at least $\gamma d_L |S|$ distinct neighbors in $R$ for some small constant $\eta > 0$.
We say that $G$ is a \emph{$\gamma$-two-sided vertex expander} if, additionally, every small subset $S$ of the right vertices has at least $\gamma d_R |S|$ neighbors.
When $G$ achieves the golden standard of $\gamma \approx 1$, in we say that it is a \emph{lossless vertex expander}.

A key motivation to study vertex expanders is for the construction of good error-correcting codes.
The seminal work of Sipser and Spielman~\cite{SS96} showed that from any one-sided lossless expander, one can construct a good binary error-correcting code with a linear time decoding algorithm.
In the quantum setting, the work of Lin \& Hsieh~\cite{LH22b} showed that \emph{two-sided lossless expanders} with appropriate algebraic structure can be used to construct good quantum low density parity check codes.

The above applications actually go through a weaker property than lossless expansion.
Sufficiently strong vertex expansion implies \emph{unique-neighbor expansion}, the condition that every small set $S$ has many \emph{unique-neighbors}, or vertices with exactly one edge to $S$. 
We say $G$ is a \emph{$\gamma$-one-sided unique-neighbor expander} if every subset $S$ of $L$ of size at most $\eta|L|$ has at least $\gamma d_L |S|$ unique-neighbors in $R$, and that $G$ is a \emph{$\gamma$-two-sided unique neighbor expander} if additionally every subset $S$ of $R$ of size at most $\eta|R|$ has at least $\gamma d_R |S|$ unique-neighbors in $L$.
Indeed, the above works show that any graph with $>\frac12$-one-sided unique-neighbor expansion yields good classical codes, and $>\frac56$-two-sided unique-neighbor expansion (and an algebraic property) yields good quantum LDPC codes. 

In this work, we focus on the setting of two-sided vertex expanders, and specifically the task of constructing such objects. 

\parhead{Where are the vertex expanders?}
There are a plethora of constructions of spectral expanders, so it is natural to wonder whether one can obtain vertex expanders from them.
Kahale~\cite{Kah95} proved that a Ramanujan graph, i.e., a graph with optimal spectral expansion, is a $\frac{1}{2}$-two-sided vertex expander, and demonstrated a near-Ramanujan graph on which this is tight.
Unfortunately, $\frac{1}{2}$-two-sided vertex expansion falls just short of giving any unique-neighbors: in fact Kamber \& Kaufman~\cite{KK22} demonstrated that the algebraic Ramanujan graph construction of Morgenstern~\cite{Mor94} contains sublinear-sized sets with \emph{zero} unique-neighbors.

On the other hand, a random biregular graph is a two-sided lossless expander with high probability (e.g.~\cite[Theorem 4.16]{HLW06}).
However, the work of Kunisky \& Yu~\cite[Section 4.6]{KY24} gives hardness evidence that there is no efficient algorithm to certify that a random graph has unique-neighbor expansion, and in particular, to certify that a random graph has $>\frac{1}{2}$-two-sided vertex expansion, suggesting that there is no simple ``algorithmic handle'' for strong vertex expansion, such as the eigenvalues of some simple matrix.
This motivates studying constructions with more ``structure''.

\subsection{Our results}
In this work, we construct explicit $\frac35$-two-sided vertex expanders, breaking the spectral barrier.
In fact, we prove something stronger: our graphs actually have $\frac35$-two-sided unique-neighbor expansion.
They additionally have an algebraic property relevant for constructing quantum codes.
While the expansion is not quite enough to instantiate the qLDPC codes of~\cite{HL22}, which demand $\frac56$-two-sided unique-neighbor expansion, we believe this is a large step in the right direction.

\begin{mtheorem} \label{thm:main-informal}
    For any $\eps > 0$ and $\beta \in (0, 1]$, there is a large enough $d(\eps,\beta)$ such that for all ${d}_L, {d}_R \ge d(\eps,\beta)$ with $\frac{{d}_L}{{d}_R} \in [\beta, \beta+\eps]$, there is an explicit infinite family of $(5{d}_L, 5{d}_R)$-biregular graphs $(Z_m)_{m\ge 1}$ with $\left( \frac35 - \eps \right)$-two-sided unique-neighbor expansion.
\end{mtheorem}

\begin{remark}
    Our construction $Z$ on vertex set $(L,R)$ can be verified to satisfy the following \emph{algebraic property}, relevant in the context of constructing quantum codes from vertex expanders via~\cite{HL22}:
    \emph{There is a group $\Gamma$ of size $\Omega(|L| + |R|)$ that acts on $L$ and $R$ such that $gv = v$ iff $g$ is the identity element, and $\{ gu, gv \}$ is an edge iff $\{ u, v \}$ is an edge in $Z$.}
\end{remark}

We use the same \emph{tripartite line product} construction as in \cite{HMMP24}, which consists of a large tripartite base graph and a constant-sized gadget graph.
In~\cite{HMMP24}, the base graph was constructed using explicit bipartite Ramanujan graphs, whereas we instantiate it using the face-vertex incidence graphs of the Ramanujan clique complex of~\cite{LSV05,LSV05b}; see also the works of Ballantine~\cite{Bal00}, Li~\cite{Li04}, Cartwright--Sol\'{e}--\.{Z}uk~\cite{CSZ03} and Sarveniazi~\cite{Sar04}.
See \Cref{sec:overview} for an overview of our analysis and the improvement over~\cite{HMMP24}.

In service of proving \Cref{thm:main-informal}, we derive bounds on the triangle density of small sets in the Ramanujan complex of~\cite{LSV05,LSV05b}, which is of independent interest in the study of high-dimensional expanders.
In particular, we employ the $4$-dimensional Ramanujan complex in our construction, and state the triangle bounds for the $4$-dimensional case below.
\begin{lemma}[Triangle density bound in $4$D Ramanujan complex, informal]
    Let $X$ be any $\bbF_q$-Ramanujan complex on $n$ vertices, and let $U\subseteq X(0)$ be a subset of vertices of size at most $\delta n$ where $q^{-25/4} > \delta > 0$.
    The number of size-$5$ faces with $3$ or more vertices in $U$ is at most $O(q^{13/2}\cdot |U|)$.
\end{lemma}
We refer the reader to \Cref{lem:triangle-bound-m} for the more general setting, and the proof is given in \Cref{sec:triangles}.

\subsection{Related work}

Unique-neighbor expanders were first constructed by Alon \& Capalbo~\cite{AC02}, who gave several constructions, one of which involves taking a \emph{line product} of a large Ramanujan graph with the $8$-vertex, $3$-regular graph obtained by the union of the octagon and edges connecting diametrically opposite vertices.

Another construction given in this work was a one-sided unique-neighbor expander of aspect ratio $22/21$, which was extended by recent work of Asherov \& Dinur~\cite{AD23} to obtain one-sided unique-neighbor expansion for arbitrary aspect ratio. These constructions were obtained via taking the \emph{routed product} of a large biregular Ramanujan graph and a constant size random graph.
These constructions were simplified by follow-up work of Kopparty, Ron Zewi \& Saraf~\cite{KoppartyRZS24}.

The work of Capalbo, Reingold, Vadhan \& Wigderson~\cite{CRVW02} constructed one-sided lossless expanders with large degree and arbitrary aspect ratio via a generalization of the \emph{zig-zag product}~\cite{RVW00}. More recently, Golowich~\cite{Gol23} and Cohen, Roth \& Ta-Shma~\cite{CohenRTS23} gave a much simpler construction and analysis of one-sided lossless expanders based on the routed product. These routed product constructions fundamentally fall short of achieving two-sided expansion for linear size sets.

In the way of explicit constructions of two-sided vertex expanders, the work of Hsieh, McKenzie, Mohanty \& Paredes~\cite{HMMP24} constructs explicit $\gamma$-two-sided unique neighbor expanders for an extremely small positive constant $\gamma$.
Their construction also guaranteed two-sided lossless expansion for sets of size at most $\exp(O(\sqrt{\log n}))$.
This was improved to two-sided lossless expansion for polynomial sized sets by~\cite{Chen2024}.

A parallel line of works~\cite{TaShmaUZ01,GUV09} construct one-sided lossless expanders in the case where the left side is polynomially larger than the right.
In the same unbalanced setting, the recent work of~\cite{CGRZ2024} show that the one-sided lossless expanders of~\cite{KalevT22} that are based on multiplicity codes~\cite{KoppartySY14} are in fact two-sided lossless expanders.
The setting of polynomial imbalance is of interest in the literature on randomness extractors, but are not known to give good quantum LDPC codes via~\cite{LH22b}.

\subsection{Technical overview}
\label{sec:overview}

In this section, we give a brief overview; see \Cref{sec:proof-overview} for a more detailed overview of the analysis, once more notation and context has been set up.

Our construction follows the \emph{tripartite line product} framework of~\cite{HMMP24}.
The first ingredient is a large tripartite base graph $G$ on vertex set $L \cup M \cup R$ (denoting left, middle, and right vertex sets), where we place a $(k, D_L)$-biregular graph $G_L$ between $L$ and $M$, and a $(D_R, k)$-biregular graph $G_R$ between $M$ and $R$.
The second ingredient is a constant-sized gadget graph $H$, which is a $(d_L, d_R)$-biregular graph on vertex set $[D_1] \cup [D_2]$.
The tripartite line product between $G$ and $H$ is the $(k d_L, kd_R)$-biregular graph $Z$ on $L$ and $R$ obtained as follows: for each vertex in $M$, place a copy of $H$ between the $D_L$ left neighbors of $v$ and the $D_R$ right neighbors of $v$ (see \Cref{def:construction}).

Since $H$ has constant size, we can find an $H$ that satisfies strong expansion properties by brute force.
It is thus convenient to view $H$ as a random biregular graph.
The bipartite graphs $G_L$ and $G_R$ of the base graph are chosen to be appropriate bipartite expanders.
In \cite{HMMP24}, they are chosen to be explicit near-Ramanujan bipartite graphs, while in our case we set them to be the vertex-face incidence graphs from high-dimensional expanders (\Cref{sec:base-graph}), which give us additional structure.
Specifically, we choose them to be vertex-face incidence graphs of the $4$D Ramanujan complex of \cite{LSV05,LSV05b}, and in particular, $G_L$ and $G_R$ are $(5,D_L)$-biregular and $(D_R, 5)$-biregular respectively.
See \Cref{sec:construction} for specific properties of the base graph that we need.

For a set $S \subseteq L$, we would like to lower bound the number of its unique-neighbors (in $R$) in the final graph $Z$.
The analysis starts by considering $U = N_{G_L}(S) \subseteq M$.
Due to the ``randomness'' of the gadget $H$ (\Cref{def:pr-gadget}), we expect that within the gadget for each $u\in U$, almost all right-neighbors are unique-neighbors, i.e., $\approx d_L \cdot \deg_S(u)$ unique-neighbors, where $\deg_S(u) = |S \cap N_{G_L}(u)|$, as long as $|S\cap N_{G_L}(u)|$ is sufficiently small.

Therefore, we need to show that a large fraction of $u \in U$ has $\deg_S(u)$ below some threshold.
We call this the \emph{left-to-middle} analysis.
Akin to \cite{HMMP24}, we split $U$ into $U_{\low}$ (low-degree) and $U_{\high}$ (high-degree) and argue that $e(S, U_{\high})$ is small.
In our case, we show that $G_L$ satisfies \emph{triangle expansion} (\Cref{def:triangle-expansion}): the property that for any small $U' \subseteq M$, there are very few vertices in $L$ with $3$ or more edges to $U'$.
Applying this to $U_{\high}$ shows that most vertices in $S$ have at most $2$ edges to $U_{\high}$ and at least $k-2$ edges to $U_{\low}$. Then, barring collisions in $R$ arising from different $u$, we have that most vertices in $S$ have $\approx (k-2) d_L$ unique neighbors within gadgets they are part of.

Next, we need to argue that the unique-neighbors in gadgets corresponding to different $u\in U$ do not have too many collisions in $G_R$.
We call this the \emph{middle-to-right} analysis.
For example, suppose a vertex $r \in R$ is a unique-neighbor within gadget $H_u$, if $r$ is also a neighbor within $H_{u'}$ for some other $u' \in U$, then it is \emph{not} a unique-neighbor of $S$ in the final graph $Z$.
If there is a collision between $u, u' \in U$, there must be a path between $u \to r \to u'$ in $G_R$.

To bound the number of collisions, we construct a multigraph $C$ on $U$ placing an edge for each length-$2$ path $u\to r\to u'$ arising from a collision.
Based on spectral properties of the Ramanujan complex (specifically \emph{skeleton expansion}; see \Cref{def:skeleton-expansion,lem:bipartite-spectral-bound}), we can bound the number of edges inside $\ul{C}$: the \emph{simple} version of $C$ obtained by replacing every multiedge by a single edge.
To control the number of edges in $C$, we need to show that ``not too many'' edges occur with ``abnormally high'' multiplicity.

To prove the statement about multiplicities, we exploit the fact that for a pair of vertices $u,u'\in M$, the set of its common neighbors is highly constrained in the graph arising from the Ramanujan complex (\Cref{def:structured-bipartite-graph}), and crucially, the property that neighborhoods of small sets in $H$ are ``spread out'': i.e. for a vertex $u\in M$, and any small subset of its left neighbors $S_u$, the neighborhood of $S_u$ in the gadget graph $H$ does not place ``too many'' vertices on the neighbors of any fixed $u'$; this is articulated in \Cref{def:pr-gadget}.

\parhead{Why does using HDX do better than \cite{HMMP24}?}
\cite{HMMP24} use near-Ramanujan bipartite graphs as the base graph.
In their middle-to-right analysis, via sharp density bounds of small subgraphs in bipartite spectral expanders, they bound the number of collisions each vertex $u\in M$ partakes in by $\sqrt{D_R}$.
As a result, they require the gadget degrees $d_L, d_R$ to be larger than $\sqrt{D_L}, \sqrt{D_R}$ by at least some constant factor, which hurts unique-neighbor expansion within gadgets corresponding to $u\in M$ such that $\deg_S(u) \approx \sqrt{D_L}$.

One of the key properties in the Ramanujan complex we use is: the square of the graph $G_R$ restricted to $M$ looks like $\ell$-copies of an almost-Ramanujan graph of degree $D_R/\ell$, for some ``reasonably large'' $\ell$.
As an upshot, the number of \emph{other vertices} $u'\in M$ that $u$ has a collision with is at most $\sqrt{D_R/\ell}$.
We can use this to show that as long as we choose our gadget degree $d_R \gg \sqrt{D_R/\ell}$, we can prove that for a typical vertex $u\in M$, only a small fraction of the unique-neighbors within its gadget encounters a collision.
The win over the approach of \cite{HMMP24} comes from the ability to choose $d_R$ such that $\sqrt{D_R/\ell} \ll d_R \ll \sqrt{D_R}$.
This improves the range of values of $\deg_S(u)$ for which the gadget corresponding to vertex $u$ experiences lossless expansion.

\parhead{Going beyond $\frac{3}{5}$.}
For our construction, we choose $k=5$ which gets us an expansion factor of $\frac{k-2}{k} = \frac{3}{5}$ in \Cref{thm:main-informal}. 
One might ask if we could get a better expansion factor by choosing a larger $k$. Unfortunately, our analysis requires us to balance certain parameters, and we were not able to show that larger $k$ satisfies the necessary inequalities.
See \Cref{rem:parameter-requirements} for a discussion on parameter choice.

We mention a few candidate ways to extend our analysis beyond $\frac35$. First, if one could improve the bounds on triangle density in small sets, that may allow larger $k$ to satisfy the necessary inequality, thus improving the bound to $\frac{k-2}{k}$. One could also attempt to bound \emph{tetrahedron expansion} (or even larger faces) in place of triangle expansion, in hopes of satisfying the necessary inequalities for appropriately larger $k$. However, in both approaches, a key difficulty seems to be that there are very few bounds known for the incidences between planes of various dimensions, and those known are extremely weak for the size of sets that naturally arise in the links of the Grassmanian clique complex.

A concrete question in this direction is whether we can obtain  tight subgraph density bounds for the bipartite incidence graph between dimension $i$ and $j$ subspaces of $\bbP(\bbF^k)$, for sets of size $q^{i(d-i)/2}$ and $q^{j(d-j)/2}$.
As we will see in \Cref{sec:triangles}, we need such tight bounds as these graphs appear as ``average'' links in the Ramanujan complex.
\subsection{Notation}
\label{sec:notation}
We now establish some notational conventions we follow throughout the paper.

For an $n$-vertex graph $G$, we use $A_G$ to denote the adjacency matrix of $G$, and write its eigenvalues in descending order $\lambda_1(G)\ge \dots \ge \lambda_n(G)$.
We say an \emph{eigenvalue of $G$} to mean \emph{eigenvalue of the adjacency matrix of $G$}.
When $G$ is bipartite on left vertex set $A$, right vertex set $B$, and edge set $E$, we write it as $(A,B,E)$.
We use $G^{\top}$ to denote the bipartite graph $(B,A,E)$. Sometimes, when the edge set is clear, we will simply denote $G$ by $(A, B)$.

We mildly deviate from standard notation for convenience, and define $[k] \coloneqq \{ 0, 1, \dots, k-1 \}$.
\section{Explicit construction of \texorpdfstring{$3/5$}{3/5}-two-sided unique-neighbor expanders}
\label{sec:main-proof}

In this section, we prove our main theorem.
\begin{theorem}[Formal \Cref{thm:main-informal}]  \label{thm:main-formal}
    For any $\beta\in(0,1]$ and $\eps > 0$, there exists $d_0 \in \N$ such that for any $d_L,d_R\ge d_0$, there is an infinite family $(Z_m)_{m\ge 1}$ of $(5d_L, 5d_R)$-biregular bipartite graphs on $(L, R)$ such that $d_R/d_L \in [\beta, \beta+\eps]$ for which $Z_m$ is a $(3/5 - \eps)$-unique-neighbor expander.
    Further, there is an algorithm that takes in $n$ as input, and in $\poly(n)$-time constructs some $Z_m$ from this family for which $|V(Z_m)| = \Theta(n)$.
\end{theorem}

In \Cref{sec:construction}, we describe our construction and state a more general result (\Cref{thm:main-UNE}), which directly implies \Cref{thm:main-formal}.
In \Cref{sec:proof-overview}, we give a proof overview for \Cref{thm:main-UNE}, and we formally prove \Cref{thm:main-UNE} in \Cref{sec:UNE-proof}.

\subsection{Construction}
\label{sec:construction}

The construction is based on the ``tripartite line product'' in \cite[Definition 7.6]{HMMP24}, instantiated with more structured base and gadget graphs.
\begin{definition}[Construction]    \label{def:construction}
    The ingredients for our construction are:
    \begin{itemize}
        \item Two bipartite graphs $G_L = (L, M, E_L)$ and $G_R = (R, M, E_R)$, where $G_{L}$ is $k$-regular on $L$ and $D_L$-regular on $M$, and the graph $G_{R}$ is $k$-regular on $R$, and $D_R$-regular on $M$.
        For each vertex $v$ in $M$, we order its neighbors in $L$ and $R$ according to injective functions $\LNbr_{v}:[D_L]\to L$ and $\RNbr_{v}:[D_R]\to R$ respectively. 
        
        % If $G_L$ and $G_R$ support a free group action, then the order $\LNbr_v$ and $\RNbr_v$ can be chosen to be consistent with the free group action. \rachel{hmm need to explain better}
        \item A constant-sized \emph{gadget graph} $H$, which is a bipartite graph with left vertex set $[D_L]$, and right vertex set $[D_R]$.
        The graph $H$ is $d_L$-regular on the left, and $d_R$-regular on the right.
    \end{itemize}
    The final construction, which we call $Z$, is a bipartite graph on $(L, R)$ constructed by taking each vertex $u\in M$, and placing a copy of $H$ between the left and right neighbors of $u$, which we will denote as $H_u$.
    More concretely, for every $u\in M$, $i\in [D_L]$, and $j\in [D_R]$, we place an edge between $\LNbr_{u}(i)$ and $\RNbr_u(j)$ if there is an edge between $i$ and $j$ in $H$.
    We emphasize that $M$ is only used in the construction of $Z$ and does not appear in the final graph.
    % \ryan{I remember being sufficiently confused by the fine details of this that we should take care, in the end, to define this carefully.}\ryan{Another question: does the finally constructed bipartite graph have parallel edges (possibly)?  I mean, it doesn't matter much, but probably it's safest if we say ``It might have'', right?}
    % \sid{How do you feel about the above definition?} \rachel{might be reasonable to emphasize that $M$ is not part of the final graph. Something along the lines of ``We emphasize that $M$ is used only in the construction of $Z$ and does not appear in the final graph.''}
\end{definition}

\parhead{Choice of base graph.}
We choose $G_{L}$ and $G_{R}$ as (truncated) bipartite vertex-face incidence graphs of the Ramanujan complex.
We state the relevant properties of the base graph we use below, and defer the proof that such a base graph indeed exists to \Cref{sec:base-graph}.
\begin{definition}[Structured bipartite graph]  \label{def:structured-bipartite-graph}
    A $(k,D)$-biregular \emph{structured bipartite graph} $G=(V, M, E)$ is a bipartite graph between $V$ and $M$ where:
    \begin{enumerate}
        \item Every vertex in $V$ is degree-$k$ and every vertex in $M$ is degree-$D$.
        \item For each vertex $u\in M$, there is an ordering of its neighbors specified by an injective function $\Nbr_u : [D] \to V$.
        \item \label{property:M-partition} The set $M$ can be expressed as a disjoint union $\sqcup_{a\in[k]}M_a$ such that each $v\in V$ has exactly one neighbor in each $M_a$.
        \item \label{property:special-sets} For each pair of distinct $a,b\in[k]$, there exists $s_G(a,b)$ (abbreviated to $s$) such that there are $s$ \emph{special sets} $(A_i\subseteq[D])_{i\in[s]}$ of size between $\frac{D}{2s}$ and $\frac{2D}{s}$, such that for any $u\in M_a$ and $v\in M_b$, we have $N(u)\cap N(v)$ is either empty, or is equal to $\Nbr_u(A_i)$ for some $i\in [s]$.
    \end{enumerate}
\end{definition}

As a prelude to \Cref{sec:base-graph}, we will construct structured bipartite graphs $(V,M,E)$ to be the incidence graph between vertices and $(k-1)$-faces in the Ramanujan clique complex of \cite{LSV05,LSV05b}, where we set $V$ to be the set of $(k-1)$-faces and $M$ to be the vertex set.
Then, the properties listed in \Cref{def:structured-bipartite-graph} will be satisfied naturally; see \Cref{thm:LSV}.

We now describe some quantities associated to a $(k, D)$-biregular structured bipartite graph $G=(V,M,E)$ that are of interest in the analysis.
\begin{definition}[Small-set triangle expansion] \label{def:triangle-expansion}
    We say that $G$ is a \emph{$\tau$-small-set triangle expander} if for some small constant $\eta > 0$, depending on $k$ and $D$, and every $U\subseteq M$ of size at most $\eta |M|$, the number of vertices $v\in V$ with 3 or more neighbors into $U$ is at most $\tau\cdot |U|$.
\end{definition}

\begin{definition}[Small-set skeleton expansion] \label{def:skeleton-expansion}
    Let $\wt{G}$ be the simple graph obtained by placing an edge for every $u,v\in M$ such that there is at least one length-$2$ walk $(u,a,v)$ in $G$ for $a\in V$.
    We say that $G$ is a \emph{$\lambda$-small-set skeleton expander} if for some small constant $\eta > 0$, and every set $U\subseteq M$ of size at most $\eta|M|$, the largest eigenvalue of $\wt{G}[U]$ is at most $\lambda$.
\end{definition}

% \begin{definition}[Free group action]
%     For a group $\Gamma$, we say that $G$ has a \emph{free group action} if there is a group $\Gamma$ such that 
    
%     \rachel{is this the right place?}
% \end{definition}

In the following, we define notation for our construction.

\begin{notation}[$G_L, G_R, D_L, D_R, k, \tau, \lambda, s_L, s_R$]   \label{notation:base-graph-notation}
    We choose our tripartite base graph on $(L,M,R)$ with the following two structured bipartite graphs: $G_L = (L, M, E_L)$, which is $(k,D_L)$-biregular, and $G_R = (R, M, E_R)$, which is $(k,D_R)$-biregular.
    Let $\tau$ and $\lambda$ be constants such that both $G_L$ and $G_R$ are $\tau$-small set triangle expanders, and $\lambda$-small-set skeleton expanders.
    We use $s_L(a, b)$ and $s_R(a,b)$ to refer to $s_{G_L}(a,b)$ and $s_{G_R}(a,b)$ respectively.
\end{notation}

In \Cref{sec:base-graph}, we prove the following about the existence of base graphs $G_L$ and $G_R$.
Concretely, the below statement follows from \Cref{lem:basically-arbitrary-degree}.
\begin{lemma}   \label{lem:base-graph}
    Given integers $n_0$, $k$, prime power $q$, integers $D_L$ and $D_R$ that are multiples of $k!$ and have magnitude at most $c\cdot q^{\binom{k}{2}}$ for some small constant $c > 0$ as input, there is a $\poly(n_0)$-time algorithm that constructs $(L,M,R)$ where $|M| = n = \Theta(n_0)$, and $|R| = |L| \cdot D_L / D_R$, and outputs structured bipartite graphs $G_L$ on $(L,M)$ and $G_R$ on $(M,R)$ such that:
    \begin{align*}
        % \tau &= O_k(1) \cdot \max_{0 \le i_0 < i_1 < i_2 < k} \min_{(i, j) \in \left\{ \substack{
        %     (i_1 - i_0, i_2 - i_0), \\
        %     (i_2 - i_1, k + i_0 - i_1), \\
        %     (k + i_0 - i_2, k + i_1 - i_2) 
        % } \right\} } \sqrt{q^{(j-i)(k-(j-i)-1)}} \cdot q^{i(k-i)/4} \cdot q^{j(k-j)/4} \cdot q^{\binom{i}{2} + \binom{j-i}{2} + \binom{k-j}{2}} \\
        & D_L,\ D_R = \Theta_k(1) \cdot q^{\binom{k}{2}} \mcom \\
        &\tau = O_k(1) \cdot q^{\binom{k}{2} - \frac{k^2}{8} - \frac{1}{2}} \cdot \max_{0 \le i_0 < i_1 < i_2 < k} \min_{(i, j) \in \left\{ \substack{
            (i_1 - i_0, i_2 - i_0), \\
            (i_2 - i_1, k + i_0 - i_1), \\
            (k + i_0 - i_2, k + i_1 - i_2) 
        } \right\} } q^{\frac{1}{8}((i-j+2)^2 + (k-i-j)^2)} \mcom \\
        &\lambda = O_k(1) \cdot q^{\frac{1}{2}\floor{\frac{k^2}{4}}} \mcom \\
        & s_L(a,b),\ s_R(a,b) \in \bracks*{q^{k-1}, O(q^{\floor{k^2/4}})} \quad \forall a < b\in [k] \mper
    \end{align*}
    % Additionally, there is a group $\Gamma$ of order $|\Gamma| = |M|$ such that $G_L$ and $G_R$ both have a free $\Gamma$-action, such that the action of $\Gamma$ on $M$ is the same in both $G_L$ and $G_R$.
    % Furthermore, for any $g\in\Gamma$, $v\in M$, and $i\in[D_L]$, we have $g\LNbr_v(i) = \LNbr_{gv}(i)$.
    % Analogously, for $i\in [D_R]$, we have $g\RNbr_v(i) = \RNbr_{gv}(i)$.
\end{lemma}

Specifically for $k=5$, we have $D_L,D_R = \Theta(q^{10})$, $\tau = O(q^{6.5})$ (see \Cref{cor:F53-bound}), $\lambda = O(q^3)$, and $s_L, s_R \in [q^4, O(q^6)]$.

\parhead{Choice of gadget graph.}
We choose $H$ as a constant-sized lossless expanders with some pseudorandom properties, whose motivation will be clearer in the analysis.
One should think of $H$ as being a random graph; the precise properties we need are articulated in \Cref{def:pr-gadget}.
% , and \Cref{lem:pseudorandom-gadget-exists} establishes that random graphs indeed have the desired properties.
\begin{definition}[Pseudorandom gadget]
\torestate{ \label{def:pr-gadget}
    Let $H$ be a $(d_L, d_R)$-biregular bipartite graph on vertex set $([D_L], [D_R])$.
    For each $a,b\in[k]$, let $(A_i\subseteq [D_R])_{i\in[s]}$ be the \emph{special sets} from \Cref{def:structured-bipartite-graph} where $s = s_{R}(a,b)$.
    Define $D\coloneqq D_L + D_R$.
    We say $H$ is a \emph{pseudorandom gadget} if for every $a,b\in [k]$, we have the following properties:
    \begin{enumerate}
        \item \label{property:bucket-spread} For every $S\subseteq [D_L]$ such that $|S| \leq D_R/d_L$ and for every $W\subseteq[s]$ with $|W|\ge \frac{s\log D}{d_L}$,
        $$\sum_{i\in W}|N(S)\cap A_i| \le 32|W| \cdot \max\braces*{ \frac{1}{r}\cdot d_L |S|,\, \log D }\mper$$
        % \item for every $S\subseteq [D_R]$ such that $|S| \ll D_L/d_R$ and for every $W\subseteq[r]$ with $|W| \ge \frac{r\log D}{d_R}$,
        % $$\sum_{i\in W}|N(S)\cap A_i| \le 32|W| \cdot \max\braces*{ \frac{1}{r}\cdot d_R\cdot |S|,\, \log D }\mper$$
        \item \label{property:lossless} For any $S\subseteq [D_L]$ with $|S| = o_D(1) \cdot D_R/d_L$, we have $|N(S)| \ge (1-o_{D}(1)) d_L |S|$.
    \end{enumerate}
}
\end{definition}

\begin{definition}[Good gadget]
    A \emph{good gadget} graph $H$ is a $(d_L, d_R)$-biregular graph on $([D_L], [D_R])$ such that $H$ and $H^{\top}$ are both pseudorandom gadgets.
\end{definition}

\begin{lemma} \label{lem:pseudorandom-gadget-exists}
    For $d_L, d_R \ge \log^2 D$, and $d_R = o_D(1)\cdot D_L$, there exists a good gadget graph $H$.
\end{lemma}

In \Cref{sec:random-gadget}, we will prove \Cref{lem:pseudorandom-gadget-exists} by showing that a random gadget satisfies the desired properties with high probability.

\parhead{Combining the parts.} Finally, we state our main theorem below.

\begin{theorem} \label{thm:main-UNE}
    Suppose for some $\delta > 0$, we have:
    \begin{gather*}
        \frac{1}{\lambda} \leq \delta \leq o_D(1) \leq \frac{1}{2k} \mcom
        \quad
        % d_L, d_R\in \bracks*{\max\braces*{\lambda, \sqrt{s_L(a,b)}, \sqrt{s_R(a,b)}} \log^2 D,\  \frac{\delta D}{\tau \log D}} \mcom  \\
        \max_{a,b\in [k]} \braces*{\lambda, \sqrt{s_L(a,b)}, \sqrt{s_R(a,b)}} \frac{\log^2 D}{\delta} \leq d_L, d_R \leq \frac{\delta D}{\tau \log D} \mcom  \\
        \lambda \le \delta^2 \cdot \min_{\substack{a,b\in[k]}}\braces*{ s_L(a,b), s_R(a,b) } \mper
    \end{gather*}
    Let $Z$ be a  $(kd_L, kd_R)$-biregular graph instantiated according to \Cref{def:construction} with a base graph satisfying the properties listed in \Cref{notation:base-graph-notation}, and a good gadget graph.
    There is a constant $\eta > 0$ such that
    \begin{itemize}
        \item for every subset $S\subseteq L(X)$, where $|S|\le \eta|L(X)|$, $S$ has $(1 - \delta - o_D(1))\cdot (k-2) \cdot d_L |S|$ unique-neighbors in $Z$, and

        \item for every subset $S\subseteq R(X)$ where $|S|\le \eta|R(X)|$, $S$ has $(1 - \delta - o_D(1))\cdot(k-2)\cdot d_R |S|$ unique-neighbors in $Z$.
    \end{itemize}
    % Furthermore, there is some group $\Gamma$ of size $O(|Z|)$ for which $Z$ has a free $\Gamma$-action. \rachel{added this line in -- need to prove}\sid{Should only be true when the base graph has a group action (which it does, but currently the assumptions don't capture this.}
\end{theorem}

We are now ready to prove \Cref{thm:main-formal}.
\begin{proof}[Proof of \Cref{thm:main-formal}]
    The statement immediately follows by instantiating \Cref{thm:main-UNE} with: (1) the base graph from \Cref{lem:base-graph} with $k = 5$, $q$ large enough so that $d_L,d_R = \Theta(q^{3.25})$, and $D_L,D_R = \Theta(q^{10})$, and (2) the gadget graph from \Cref{lem:pseudorandom-gadget-exists}.
    By \Cref{lem:base-graph}, the base graph is efficiently constructible with $\tau = O(q^{6.5})$ and $\lambda = O(q^3)$ (see \Cref{cor:F53-bound} for the bound on $\tau$), and $q^4 \leq s_L(a,b), s_R(a,b) \leq O(q^6)$ for all $a,b \in [k]$.
    The gadget graph is also efficiently constructible via a brute force search over all graphs on $([D_1],[D_2])$.
    Finally, the tripartite line product can also be computed efficiently.
\end{proof}

% \begin{remark}[Free Group Action]
%     When both $G_L$ and $G_R$ have a free $\Gamma$-action that coincides on $M$ (as is the case for the graphs constructed in \Cref{lem:base-graph}), then $Z$ can be made to also have a free $\Gamma$-action, as follows. First, we should choose the neighbor orderings $\LNbr_v$ and $\RNbr_v$ to be consistent with the $\Gamma$ action: for $g \in \Gamma$ then $g \LNbr_v(i) = \LNbr_{gv}(i)$. Then, since the gadget graph is placed the same way on the neighborhood of each $v \in M$, the resulting edges in $Z$ also respect the group action. 
% \end{remark}

\subsection{Proof overview of \texorpdfstring{\Cref{thm:main-UNE}}{Theorem~\ref{thm:main-UNE}}}
\label{sec:proof-overview}

For \Cref{thm:main-UNE}, we will only prove the unique-neighbor expansion for $S \subseteq L(X)$; the argument for $S \subseteq R(X)$ is exactly the same.
Similar to \cite{HMMP24}, we split the proof of \Cref{thm:main-UNE} into two parts: left-to-middle and middle-to-right.
Given a set $S \subseteq L(X)$, let $U \subseteq M$ be the neighbors of $S$ in $M$ in the base graph.

\parhead{Left-to-middle analysis.}
We divide $U$ into $U_{\ell}$ (low-degree) and $U_{\high}$ (high-degree), and we would like to prove that a large (close to $\frac{k-2}{k}$) fraction of edges leaving $S$ go into $U_{\low}$.
This is desirable because in our gadget $H$ (satisfying \Cref{def:pr-gadget}), small subsets (of $[D_1]$) expand losslessly while we have no guarantees on large subsets.

To do so, we apply the small-set triangle expansion (\Cref{def:triangle-expansion}) of the base graph $G_L$ to $U_{\high}$, which roughly states that very few vertices in $S$ have more than $2$ edges to $U_{\high}$.
Since each vertex in $S$ has degree $k$, this implies that most vertices in $S$ have at least $k-2$ edges going to $U_{\low}$, thus giving a lower bound on $e(S, U_{\low})$.
This is articulated in \Cref{lem:edges-into-Ulow}.

\parhead{Middle-to-right analysis.}
We need to argue that the unique-neighbors from each $u\in U$ (within the gadget $H_u$) do not have too many collisions.
It is convenient to visualize such collisions as follows:
(1) for each $u\in U$, we draw blue edges from $u$ to all its right neighbors that are unique-neighbors of $N_L(u)$ within the gadget $H_u$;
(2) we draw red edges from $u$ to all its right neighbors that are nonunique-neighbors of $N_L(u)$ within $H_u$.
The following is a simple observation:
\begin{observation} \label{obs:un-one-blue-edges}
    The set of unique-neighbors of $S$ in the final graph $Z$ is the set of vertices in $R$ incident to exactly one blue edge and no red edge.
\end{observation}

We note that since vertices in $U_{\low}$ have low degree on the left, they will have many blue edges going to $R$ (almost $(k-2)d_L |S|$ in total).
We next need to prove that few collisions occur.

\parhead{Bounding collisions.}
We define a (multi-)graph $C$ on vertex set $U$ by placing a copy of the edge $\{u,v\}$ for each $u \in U_{\low}$, $v\in U$, $r\in R$ such that $u$ has a blue edge to $r$ and $v$ has an (either blue or red) edge to $r$.
Then, the number of collisions is exactly $e(C)$.

A collision between $u, v\in U$ can occur only if $u,v$ have common neighbors in $R$.
In other words, the graph $\ul{C}$, defined as the graph obtained by removing parallel edges in $C$, is a subgraph of $\wt{G}_R$ --- the simple graph on $M$ obtained from length-2 walks in $G_R$ (see \Cref{def:skeleton-expansion}).
The most natural attempt to bound $e(C)$ is to use the small-set skeleton expansion of $\wt{G}_R$ (\Cref{def:skeleton-expansion}), which implies that small subgraphs in $\wt{G}_R$ have small average degree (see \Cref{lem:out-degree-bound}).
However, $C$ can have multiplicities, which complicate the analysis.

Let us examine where the multiplicities may come from.
An edge $\{u,v\}$ in $C$ can have multiplicites when they share several common neighbors in $R$ in the bipartite graph $G_R$.
To analyze the common neighborhood structure, we use the fact that $G_R$ satisfies \Cref{property:special-sets} in \Cref{def:structured-bipartite-graph}, which states that the common neighborhood of $u \in M_a$ and $v\in M_b$ is either empty or equal to some``special set''.
We thus have the following crucial observation:

% Recall from \Cref{property:M-partition} of \Cref{def:structured-bipartite-graph} that $M$ is partitioned into $M_1 \cup \cdots \cup M_k$.
% In the bipartite graph $G_{R}$, each $r\in R$ is connected to a unique vertex in $M_a$ for each $a\in [k]$.
% Thus, the $D_2$ neighbors of a vertex $u\in M_a$ in $G_{R}$ can be partitioned according to vertices in $M_b$ for each $b\neq a$, and these are the \emph{special sets} defined in \Cref{property:special-sets} of \Cref{def:structured-bipartite-graph}.
% In particular, each $v \in N_{\wt{G}_{R}}(u) \cap M_b$ (i.e., vertex in $M_b$ that has a common neighbor with $u$ in $G_{R}$) corresponds to one special set (for the pair $a,b\in [k]$).

\begin{observation} \label{obs:bound-multiplicity}
    For $u\in M_a$ and $b \neq a$, consider the $s_{G_R}(a,b)$ special sets in \Cref{property:special-sets} of \Cref{def:structured-bipartite-graph} that correspond to subsets of $N_{G_R}(u) \subseteq R$.
    For any $v \in M_b$ with $\{u,v\} \in \wt{G}_{R}$ (i.e., $N_{G_R}(u) \cap N_{G_R}(v) \neq \varnothing$), the multiplicity of the edge $\{u,v\}$ in $C$ is at most the number of blue or red edges from $u$ that land in any special set.
\end{observation}

In light of this, we can utilize \Cref{property:bucket-spread} of \Cref{def:pr-gadget} of the gadget, which states that the blue or red edges must be evenly spread among the special sets.
For $u \in U_{\low} \cap M_a$, suppose it has $\deg_S(u)$ edges to $S$ on the left, then it has roughly $d_L \cdot \deg_{S}(u)$ blue or red edges going to $R$.
Suppose $u$ has $\lambda$ neighbors in $M_b$ for a fixed $b\in [k]$, which correspond to $\lambda$ of the $s_R(a,b)$ special sets.
Then, \Cref{property:bucket-spread} of \Cref{def:pr-gadget} states that the number of blue or red edges that land in the these special sets is roughly as expected: $\lambda/s_R(a,b) \cdot d_L \cdot \deg_{S}(u)$.
Moreover, by the small-set skeleton expansion of $\wt{G}_R$ (\Cref{def:skeleton-expansion}), it is true that $\ul{C}[U_{\low}]$, a subgraph of $\wt{G}_R[U_{\low}]$, has average degree at most $\lambda$.
This gives an upper bound on the total multiplicites of edges within $U_{\low}$ in $C$.

% \tim{This part is very hard to explain...}

\parhead{Collisions between $U_{\low}$ and $U_{\high}$.}
Unfortunately, we no longer have the degree bounds for edges between $U_{\low}$ and $U_{\high}$ in $\ul{C}$.
For $u \in U_{\low} \cap M_a$, if $\deg_{\ul{C}}(u \to U_{\high} \cap M_b) \leq \lambda/\delta$ (for some $\delta = o_D(1)$ not too small), then we can still use the previous argument that only 
an $o_D(1)$ fraction of the blue or red edges land in the special sets.
However, if $\deg_{\ul{C}}(u \to U_{\high} \cap M_b)$ is large, we need a better argument.

We call such vertices \emph{saturated}, denoted $\Usat$.
We no longer have upper bounds using the edge multiplicities in $C$ and the special sets.
% We only have the naive bound of $\frac{\tau}{\delta} |\Usat| \cdot d_L$ on the edges between $\Usat$ and $U_{\high}$.
We must instead directly upper bound $|\Usat|$.
Consider the bipartite graph between saturated vertices $U_{\high} \cap M_b$.
The key observation is that the maximum eigenvalue is also upper bounded by $\lambda$, and since the degrees of the saturated vertices are $> \lambda/\delta$, the average degree of $U_h \cap M_b$ in this graph is $\lesssim \frac{\lambda^2}{\lambda/\delta} \leq \delta \lambda$ (\Cref{lem:average-degree-bound}).
In particular, this implies that $\Usat$ and $U_{\high} \cap M_b$ are very unbalanced --- $|\Usat| \lesssim \delta^2 |U_{\high}|$.
On the other hand, since the vertices in $U_h$ all have large degrees to $S$ (in the graph $G_L$), we also have an upper bound on $|U_{\high}|$ in terms of $e_{G_L}(S, U) = k|S|$.
This completes the proof.

\parhead{Concrete parameters and requirements.}
To prove the triangle expansion, we analyze the $(k-1)$-dimensional Ramanujan complex of \cite{LSV05, LSV05b} (see \Cref{sec:base-graph}) and show an upper bound of $\tau |U|$ on the number of $(k-1)$-faces in the complex containing a triangle in any small vertex set $U$
(see \Cref{lem:triangle-bound-m}; the proof is carried out in \Cref{sec:triangles}).
On the other hand, the skeleton expansion $\lambda$ follows from known spectral properties of the Ramanujan complex (see \Cref{lem:skeleton-expansion}).
The concrete bounds for $D_L$, $\tau$ and $\lambda$ are stated in \Cref{lem:base-graph}.

\begin{remark}[Parameter requirements] \label{rem:parameter-requirements}
    At a high level, we have two requirements.
    Let $\delta$ be some small enough constant.
    For the left-to-middle analysis, we need to set the degree threshold for $U_{\high} \subseteq M$ to be $\tau/\delta$ in order to get meaningful bounds on $U_{\high}$.
    Thus, for the low-degree vertices to have good unique-neighbor expansion within each gadget $H$, we need $\frac{\tau}{\delta} \cdot d_L < D_L$.
    For the middle-to-right analysis, $\lambda$-expansion roughly means that any small subgraph has average degree $\lambda$, hence we need $d_L > \lambda/\delta$ (here, we ignore the multiplicities for simplicity).
    This gives
    \begin{align*}
        D_L / \tau \gg d_L \gg \lambda \mper
    \end{align*}
    Unfortunately, the above requirements restrict us to $k \leq 5$.
    Let $q$ be the large prime power used in the Ramanujan complex.
    From \Cref{lem:base-graph}, we have:
    \begin{itemize}
        \item $k=3$: $D_L = \Omega(q^{3})$, $\tau = O(q^{1.5})$, and $\lambda = O(q)$. (See also \Cref{lem:k=3-triangle-bound}.)
        \item $k=5$: $D_L = \Omega(q^{10})$, $\tau = O(q^{6.5})$, and $\lambda = O(q^3)$. (See also \Cref{cor:F53-bound}.)
        \item $k=6$: $D_L = \Omega(q^{15})$, $\lambda = O(q^{4.5})$, but the bound on $\tau$ we have is no better than $O(q^{10.5})$.
    \end{itemize}
    We can see that for $k=6$, the parameters do not satisfy the aforementioned requirements.
\end{remark}

\subsection{Proof of \texorpdfstring{\Cref{thm:main-UNE}}{Theorem~\ref{thm:main-UNE}}}  \label{sec:UNE-proof}
% We now focus on proving \Cref{thm:main-UNE}.
% Given a set $S\subseteq L(X)$, our goal is to prove that it has a large number of unique-neighbors in $R(X)$.
We first introduce some notation.
\begin{itemize}
    \item Fix a subset $S \subseteq L$ of size $\leq \eta |L|$.
    Let $U\subseteq M$ be the neighbors of $S$ in $M$ in $G_{L}$, and let $\Gamma_S = G_L[S \cup U]$ be the induced subgraph of $S, U$.
    % \item We use $\Gamma_S$ to denote the induced subgraph between $S$ and $U$.
    \item We use $U_{\high}$ (``high'') to denote the set of vertices in $U$ with degree (in $\Gamma_S$) exceeding $\tau\cdot\frac{1}{\delta}$,
    and we use $U_{\low}$ (``low'') to denote $U \setminus U_{\high}$.
    \item For each $u\in U$, we draw blue edges from $u$ to all its right neighbors that are unique-neighbors of $N_L(u)$ within the gadget $H_u$.
    We will use $\Blue(u)$ to refer to the blue edges incident to $u$.
    \item We draw red edges from $u$ to all its right neighbors that are nonunique-neighbors of $N_L(u)$ within $H_u$.
    We will use $\Red(u)$ to refer to the red edges incident to $u$.

    \item We use $\wt{G}_R$ to denote the simple graph on vertex set $M$ obtained by placing an edge for every $u,v\in M$ such that there is at least one length-$2$ walk $(u,a,v)$ in $G_R$ for $a\in R$.
\end{itemize}

\noindent We start with the left-to-middle analysis.
The following lemma states that almost $\frac{k-2}{k}$ fraction of edges leaving $S$ goes to the low-degree vertices $U_{\low}$.

\begin{lemma}   \label{lem:edges-into-Ulow}
    Suppose $\delta k \leq 1/2$.
    The number of edges in $\Gamma_S$ incident to $U_{\low}$ is at least $(1-4\delta)(k-2)|S|$.
\end{lemma}
\begin{proof}
    By definition of $U_h$, the number of edges from $S$ to $U_h$ is at least $\tau\cdot\frac{|U_h|}{\delta}$.
    Define $S_{\ge 3}$ as the vertices of $S$ with at least $3$ neighbors into $U_h$.
    By the small-set triangle expansion (\Cref{def:triangle-expansion}) of $G_L$, the number of vertices in $S_{\ge 3}$ is at most $\tau\cdot |U_{\high}|$.
    Since every vertex in $S_{\ge 3}$ is degree-$k$, we have $e(S_{\ge 3}, U_{\high}) \leq k\tau |U_{\high}|$.
    On the other hand, by definition of $U_{\high}$, each vertex in $U_{\high}$ has at least $\tau/\delta$ edges to $S$, so $e(S, U_{\high}) \geq \tau |U_{\high}| / \delta$.
    Consequently, we have
    \[
        \frac{e(S_{\ge 3}, U_{\high})}{e(S, U_{\high})} \le \delta k\mper
    \]

    Every vertex in $S\setminus S_{\ge 3}$ has at most $2$ edges into $U_{\high}$, and hence at least $k-2$ edges into $U_{\low}$.
    Thus, we have:
    \[
        e(S, U_{\low}) \ge e(S\setminus S_{\ge 3}, U_{\low}) \ge \frac{k-2}{2} e(S\setminus S_{\ge 3}, U_{\high}) \ge \frac{k-2}{2}(1-\delta k) e(S, U_h)\mper
    \]
    Since $\frac{b}{a+b}$ is monotone increasing in $b$ for $a,b>0$, we have:
    \[
        \frac{e(S,U_{\low})}{k|S|} = \frac{e(S,U_{\low})}{e(S, U_{\low}) + e(S, U_{\high})}
        \geq \frac{\frac{k-2}{2}(1-\delta k)}{\frac{k-2}{2}(1-\delta k) + 1}
        = \frac{k-2}{k} \parens*{1 - \frac{2\delta}{1-\delta(k-2)}} \mcom
    \]
    which completes the proof.
\end{proof}

\begin{lemma}   \label{lem:blue-edge-count}
    The number of vertices in $R$ incident to exactly one blue edge and no red edges is at least $(1-o_D(1))(k-2) d_L |S|$.
\end{lemma}

Before we prove the above lemma, we first show how to complete the proof of \Cref{thm:main-UNE}.
\begin{proof}[Proof of \Cref{thm:main-UNE}]
    The statement is immediate from \Cref{obs:un-one-blue-edges,lem:blue-edge-count}.
\end{proof}

We will need the following folklore facts that we will apply to $U_{\low}$.
\begin{lemma} \label{lem:out-degree-bound}
    For a graph $G$, suppose $\lambda$ is the maximum eigenvalue of the adjacency matrix, then there is an orientation of the edges in $G$ such that all vertices have outdegree at most $\lambda$.
\end{lemma}
\begin{proof}
    Let $A$ be the adjacency matrix of $G$, and let $n$ be the number of vertices.
    Note that $\lambda_{\max}(A) \leq \lambda$ implies that all principal submatrices of $A$, i.e., induced subgraphs of $G$, have maximum eigenvalue at most $\lambda$.
    Moreover, $\lambda_{\max}(A) \leq \lambda$ implies that the average degree is at most $\lambda$, since $2|E(G)| = \one^\top A \one \leq \lambda\|\one\|_2^2 = \lambda n$.
    We can thus find a vertex $v$ with degree at most $\lambda$ and orient all its incident edges to point away from $v$.
    Then, we can remove $v$ from $G$ and repeat this process, since all induced subgraphs of $G$ have average degree at most $\lambda$.
    In the end, we obtain an orientation of edges where all outdegrees are at most $\lambda$.
\end{proof}

\begin{lemma}[{\cite[Lemma 6.2]{HMMP24}}] \label{lem:average-degree-bound}
    Let $G$ be a bipartite graph with average left-degree $d_1$ and average right-degree $d_2$.
    Let $\lambda$ be the maximum eigenvalue of the adjacency matrix.
    Then, $(d_1-1)(d_2-1) \leq \lambda^2$.
\end{lemma}

Now, we prove \Cref{lem:blue-edge-count}.
\begin{proof}[Proof of \Cref{lem:blue-edge-count}]
    We will refer to a vertex $r$ in $R$ incident to exactly one blue edge and no red edges as a \emph{blue unique-neighbor}.
    By lossless expansion of the gadget $H$ (\Cref{property:lossless} of \Cref{def:pr-gadget}) and \Cref{lem:edges-into-Ulow}, the number of blue edges from $U_{\low}$ to $R$ is at least
    \begin{align*}
        \sum_{u\in U_{\low}} (1-o_D(1)) d_L \cdot \deg_{\Gamma_S}(u)
        = (1-o_D(1)) d_L \cdot e(S, U_{\low})
        \geq (1-o_D(1)) \cdot d_L (k-2)|S| \mcom
    \end{align*}
    where we use the fact that $\deg_{\Gamma_S}(u) \leq \frac{\tau}{\delta} \leq \frac{D}{d_L \log D}$ (required in \Cref{def:pr-gadget}) for all $u\in U_{\low}$.

    To control the number of blue unique-neighbors, we bound the number of blue edges to $R$ that ``collide'' with another edge using bounds on the edge density of small sets in $\wt{G}_{R}$.

    Construct a (multi-)graph $C$ on vertex set $U$ by placing a copy of the edge $\{u,v\}$ for each $u \in U_{\low}$, $v\in U$, $r\in R$ such that $u$ has a blue edge to $r$ and $v$ has an (either blue or red) edge to $r$.
    Let $\ul{C}$ be the graph obtained by removing duplicate edges from $C$.
    By \Cref{obs:un-one-blue-edges}, the number of blue unique-neighbors is at most
    \[
        (1-o_D(1))\cdot (k-2) d_L|S| - 2e(C)\mper
    \]
    It suffices to show that $e(C) \leq o_D(1) \cdot kd_L |S|$.
    To bound $e(C)$, we write it as $e_C(U_{\low}) + e_C(U_{\low}, U_{\high})$ (note that there are no edges within $U_{\high}$ in $C$), and bound each term separately.

    We first bound $e(U_{\low})$.
    By $\lambda$-small-set skeleton expansion of $\wt{G}_R$, the largest eigenvalue of $\ul{C}[U_{\low}]$ is bounded by $\lambda$.
    Consequently, by \Cref{lem:out-degree-bound}, there is an orientation of the edges of $\ul{C}[U_{\low}]$ such that all vertices have outdegree bounded by $\lambda$.
    Pick such an orientation and let $O(v)$ denote the set of outgoing edges incident to a vertex $v$.
    We can write:
    \[
        e_C(U_{\low}) = \sum_{v\in U_{\low}} \sum_{e\in O(v)} \mathrm{multiplicity}(e) \mper
    \]
    Fix $a\neq b\in [k]$.
    We will write $s \coloneqq s_R(a,b)$ for simplicity (when $a,b\in [k]$ are clear from context).
    For $v\in M_a$, define $O_b(v)$ as the set of $e\in O(v)$ directed towards a vertex in $M_b$.
    By \Cref{obs:bound-multiplicity}, we can bound $\sum_{e\in O_b(v)} \mathrm{multiplicity}(e)$ by the number of blue or red edges that land in any $|O_b(v)|$ of the special sets.
    Since $|O_b(v)| \le \lambda < \frac{\lambda}{\delta}$ (here, we use a loose bound of $\frac{\lambda}{\delta}$ on $|O_b(v)|$ for convenience later), we can apply the bound in \Cref{property:bucket-spread} of \Cref{def:pr-gadget} with $|S| = \deg_{\Gamma_S}(v)$ and $|W| = \max\{\frac{\lambda}{\delta}, \frac{s \log D}{d_L}\}$ and get
    \begin{equation}  \label{eq:bound-multlicities}
    \begin{aligned}
        % \sum_{e\in O(v)} \mathrm{multiplicity}(e) &=
        % \sum_{b = 1}^k 
        \sum_{e\in O_b(v)} \mathrm{multiplicity}(e) 
        &\le 32 \sum_{b=1}^k \max \braces*{\frac{\lambda}{\delta},\ \frac{s \log D}{d_L}} \cdot \max\braces*{ \frac{1}{s}\cdot d_L \cdot \deg_{\Gamma_S}(v),\ \log D } \\
        &\le 32\sum_{b=1}^k \max\braces*{ \frac{\lambda}{\delta s},\ \frac{\lambda \log D}{\delta d_L \deg_{\Gamma_S}(v)},\ \frac{\log D}{d_L},\ \frac{s \log^2 D}{d_L^2 \deg_{\Gamma_S}(v)} }
        \cdot d_L \cdot \deg_{\Gamma_S}(v) \\
        &\leq o_D(1) \cdot d_L \cdot \deg_{\Gamma_S}(v) \mper
    \end{aligned}
    \end{equation}
    Here, we use the assumptions on the parameters listed in \Cref{thm:main-UNE}: $\lambda \leq \delta^2 s$, and $d_L \geq \frac{1}{\delta}\max\{\lambda, \sqrt{s}\} \log^2 D \geq \log^2 D$.

    Summing over all $b\in [k]$ gives at most a $k$ factor, and $\sum_{v\in U_{\low}} \deg_{\Gamma_S}(v) \leq k|S|$.
    Thus, we have
    \begin{align*}
        e_C(U_{\low}) \leq o_D(1) \cdot k d_L |S| \mper
    \end{align*}

    We now turn our attention to showing a bound on $e(U_{\low}, U_{\high})$, which we can write as:
    \[
        \sum_{a,b\in[k]} e_C\parens*{ U_{\low} \cap M_a, U_{\high} \cap M_b }\mper
    \]
    Again, we fix $a, b$ and prove a bound on the summand.
    Call a vertex $v\in U_{\low}\cap M_a$ \emph{saturated} if $\deg_{\ul{C}}\parens*{v\to U_{\high}\cap M_b} > \frac{\lambda}{\delta}$, and call a vertex \emph{unsaturated} otherwise.
    We use $\Usat$ and $\Uunsat$ to denote the saturated and unsaturated vertices in $U_{\low} \cap M_a$ respectively.

    For any unsaturated vertex, its degree in $\ul{C}$ is at most $\frac{\lambda}{\delta}$, and thus, the exact same calculation as \Cref{eq:bound-multlicities} shows that
    \begin{align*}
        e_C(\Uunsat, U_{\high} \cap M_b) \leq o_D(1) \cdot k d_L |S| \mper
    \end{align*}
    Now, we bound the contribution of saturated vertices.
    Here, we only have the naive bound
    \begin{align*}
        e_C(\Usat, U_h \cap M_b) \leq \sum_{u\in \Usat} \deg_{\Gamma_S}(u) \cdot d_L
        \leq \frac{\tau}{\delta} \cdot d_L |\Usat| \mcom
    \end{align*}
    since $\deg_{\Gamma_S}(u) \leq \frac{\tau}{\delta}$ for all $u\in \Usat \subseteq U_{\low}$.
    It remains to bound $|\Usat|$.

    Consider the bipartite graph between $\Usat$ and $U_h \cap M_b$, within $\ul{C}$.
    We know that (i) the top eigenvalue of this bipartite graph is at most $\lambda$, by $\lambda$-small-set skeleton expansion of $\wt{G}_R$, and
    (ii) every vertex in $\Usat$ has degree at least $\frac{\lambda}{\delta}$.
    Thus, by \Cref{lem:average-degree-bound}, the average degree of $U_{\high} \cap M_b$ is at most $\delta \lambda + 1$ which is at most $2\delta \lambda$ by the assumption that $\delta \lambda \geq 1$.
    In particular, we have
    \begin{align*}
        |\Usat| \leq |U_{\high}\cap M_b| \cdot \frac{2\delta \lambda}{\lambda/\delta}
        \leq 2\delta^2 |U_{\high}| \mper
    \end{align*}
    On the other hand, since every vertex $u$ in $U_h$ has $\deg_{\Gamma_S}(u) \geq \tau$, we have
    $\tau |U_{\high}| \leq e_{\Gamma_S}(S, U_{\high}) \leq k|S|$.
    This implies that
    \begin{align*}
        e_C(\Usat, U_h \cap M_b) \leq 2\delta \cdot kd_L |S| \leq o_D(1) \cdot kd_L |S| \mcom
    \end{align*}
    since $\delta \leq o_D(1)$.
    Summing up contributions from all $a,b\in [k]$ gives an extra $k^2$ factor.

    Combining all of the above, we get that $e(C) \leq o_D(1) \cdot kd_L |S|$, completing the proof.
\end{proof}

\section{Construction of the base graph}
\label{sec:base-graph}

We first list the properties of the Ramanujan complex of \cite{LSV05, LSV05b} that we use to construct our base graph.
For standard terminology pertaining simplicial complexes, see, e.g., \cite[Section 2.4]{ALOV19}.
Moreover, we recall the definitions and approximations (for large $q$) of Gaussian binomial coefficients: $[k]_q = \frac{q^k-1}{q-1} \sim q^{k-1}$, $[k]_q! = [1]_q [2]_q \cdots [k]_q \sim q^{\binom{k}{2}}$, and $\gbinom{k}{i}_q = \frac{[k]_q!}{[i]_q![k-i]_q!} \sim q^{i(k-i)}$, and note that $\gbinom{k}{i}_q \leq \gbinom{k}{\floor{k/2}}_q \sim q^{\floor{k^2/4}}$.

\begin{definition}[Cayley clique complex]
    Let $\Gamma$ be a group, and let $S$ be a symmetric set of generators of $\Gamma$.
    We use $\calH = \Cay_k(\Gamma,S)$ to denote the $(k-1)$-dimensional complex constructed by choosing each $\{u_0,\dots,u_{k-1}\}\subseteq\Gamma$ as a face iff for every $a,b\in[k]$, $u_a u_b^{-1}\in S$.
    We refer to $\calH$ as the \emph{Cayley clique complex} on $\Gamma$ with generating set $S$.
\end{definition}

\begin{definition}[Incidence graph]
    The \emph{incidence graph} $\Inc_{\calH}$ of a simplicial complex $\calH$ is a bipartite graph between the vertices and $(k-1)$-faces of $\calH$, where there is an edge between a vertex $v$ and a face $e$ if $e$ contains $v$.
\end{definition}

\begin{definition}[Unweighted skeleton]
    The \emph{unweighted skeleton} $\Skel_{\calH}$ of a complex $\calH$ refers to the simple graph with vertex set $\Gamma$ obtained by placing an edge between $u$ and $v$ if they share a face.
\end{definition}

\begin{theorem}[{\cite{LSV05,LSV05b}}] \label{thm:LSV}
    Fix a prime power $q$ and integer $k\ge 2$.
    There is an algorithm that, on input an integer $e > 1$ (with $q^e > 4k^2 + 1$), constructs:
    \begin{itemize}
        \item The group $\Gamma = \mathrm{PGL}_k\parens*{\bbF_{q^e}}$, which has cardinality
        $n \coloneqq \frac{1}{q^e-1} \prod_{0\le i \le k} \parens*{ q^{ek} - q^{ei} }\sim q^{ek^2-1}$.
        \item A set of generators $S = S_1 \sqcup \dots \sqcup S_{k-1} \subset \Gamma$, where $\abs*{S_i} = \gchoose{k}{i}_q \sim q^{i(k-i)}$ (the Gaussian binomial coefficient, equal to the number of subspaces of dimension $k$ in $\bbF_q^k$); hence $|S| \sim q^{\floor{k^2/4}}$.
    \end{itemize}
    The algorithm runs in time $\poly(n)$.
    The resulting Cayley complex $\Cay_k(\Gamma,S)$, called the Ramanujan clique complex, is a $k$-partite $(k-1)$-dimensional simplicial complex with the following properties.
    \begin{itemize}
        \item The vertex set $V$, of size $n$, has equal-size parts $V_0,V_1,\dots,V_{k-1}$.
        \item For each $0 < i < k$, the generating set $S_i$ creates directed edges that go from $V_a$ to $V_{a+i}$ (all part indices are taken mod $k$).
        \item Let $A_i$ denote the directed subgraph on $V$ consisting of just those edges created by $S_i$, so $A_i$ is out-regular of degree $\gchoose{k}{i}_q \sim q^{i(k-i)}$.
        \item Thought of as adjacency matrices, $A_1,\dots,A_{k-1}$ commute (and are normal); thus they are simultaneously diagonalizable.
        For each of the $n$ eigenvectors, there is an ``eigentuple'' $\vec{\lambda}$ of eigenvalues $(\lambda_1,\dots,\lambda_{k-1})\in\C^{k-1}$ associated to $A_1,\dots,A_{k-1}$.
        \item We say that an eigentuple is trivial if $|\lambda_i| = \gchoose{k}{i}_q$ for all $i$.
        The eigenvectors corresponding to the trivial eigentuple take on a constant value on each part.
        \item Every nontrivial eigenvalue of $A_k$ has modulus at most $\binom{k}{i}\sqrt{q^{i(k-i)}}$.
        \item For any $i$ and any vertex $v\in V_i$, the link of $v$ can be identified with the spherical building $\bbP(\bbF_q^k)$.
        Here, $\bbP(\bbF_q^k)$ is the $(k-1)$-partite $(k-2)$-dimensional simplicial complex whose vertices are all the non-trivial subspaces of $\bbF_q^{k}$ (i.e., not $\{0\}$ and not $\bbF_q^{k}$), and a $t$-face corresponds to a chain of subspaces $W_0 \subset W_1 \subset W_2 \subset \cdots \subset W_t$.
        The $j$-dimensional subspaces exactly correspond to the neighbors of $v$ in $V_{i+j}$.
    \end{itemize}
\end{theorem}

We now use the known properties of the Ramanujan complex listed above in \Cref{thm:LSV} to prove that for any pair of parts $i$ and $j$, the bipartite graph of the $1$-skeleton of the Ramanujan clique complex is $O(1)$-Ramanujan.
\begin{lemma}  \label{lem:bipartite-spectral-bound}
    For $i, j\in [k]$ such that $i > j$, let $A$ denote the adjacency matrix of the bipartite graph between $V_i$ and $V_j$.
    We have:
    \[
        \lambda_2(A)\le 2 \binom{k}{i-j}\sqrt{q^{(i-j)(k-i+j)}}\mper
    \]
\end{lemma}
\begin{proof}
    Let $v$ be a nontrivial eigenvector of $A$.
    Observe that $(A_{i-j} + A_{j-i})v = (1+\Ind[i-j = j-i])Av + w$, where the support of $w$ is disjoint from that of $Av$. 
    Thus:
    \(
        \norm*{Av} \le \norm*{\parens*{ A_{i-j} + A_{j-i} } v}\mper
    \)

    Now, observe that for any $t\in[k]$, we have $\angles*{v_t, \vec{1}_t} = 0$ where $v_t$ and $\vec{1}_t$ denote their respective restrictions to $V_t$.
    Consequently, $v$ must be orthogonal to all the trivial eigenvectors of $A_{i-j} + A_{j-i}$, and hence must be spanned by its nontrivial eigenvectors.
    By the bound on the nontrivial eigenvalues from \Cref{thm:LSV}, we have:
    \[
        \norm*{\parens*{ A_{i-j} + A_{j-i} } v} \le 2\binom{k}{i-j}\sqrt{q^{(i-j)(k-i+j)}}\mper
    \]
    Consequently, we have $\lambda_2\parens*{A} \le 2\binom{k}{i-j}\sqrt{q^{(i-j)(k-i+j)}}$.
\end{proof}

\subsection{Properties of the spherical building}
\begin{observation}
    The $1$-skeleton of $\bbP(\F_q^k)$ is a $(k-1)$-partite graph with vertex sets $V_1, V_2,\dots, V_{k-1}$, where $V_i$ consists of $i$-dimensional subspaces in $\F_p^k$.
    Thus, $|V_i| = \gbinom{k}{i}_q$, which are the Gaussian binomial coefficients.
    For $i < j$, the bipartite graph between $V_i$ and $V_j$ in the $1$-skeleton has left and right degree $\gbinom{k-i}{j-i}_q \approx q^{(j-i)(d-j)}$ and $\gbinom{j}{i}_q\approx q^{i(j-i)}$ respectively.
\end{observation}
\begin{remark}
    For example, when $k=3$, we have $|V_1| = |V_2| = \gbinom{3}{1}_q = q^2+q+1$, and the bipartite graph between $V_1$ and $V_2$ has left and right degree $\gbinom{2}{1}_q = q+1$.
\end{remark}

\begin{lemma}[Corollary 7.1 of \cite{GHKLZ22}] \label{lem:link-second-eigenvalue}
    Let $1 \leq i < j \leq k-1$.
    The bipartite graph between $V_i$ and $V_j$ has second eigenvalue
    \begin{align*}
        \lambda_2 = \sqrt{q^{j-i} \gbinom{j-1}{j-i}_q \gbinom{k-i-1}{j-i}_q} \approx q^{(j-i)(k+i-j-1)/2} \mper
    \end{align*}
\end{lemma}

\subsection{Construction of base graph from Ramanujan complex: balanced case}
We first specify how to construct a structured bipartite graph, which is immediately useful in the ``balanced'' setting.
\begin{definition}[Construction of structured bipartite graph]  \label{def:base-graph-construction}
    We describe a procedure that takes in nonnegative integers $n_0$, $k$, and prime power $q$, and for $n = \Theta(n_0)$ and $D = |\bbP(\bbF_q^k)|$, constructs vertex sets $V,M$, along with a $(k,D)$-biregular bipartite graph $G$ on $(V,M)$.
    \begin{itemize}
        \item We find $e$ such that $n\coloneqq\abs*{\mathrm{PGL}_k\parens*{\F_{q^e}}}$ is $\Theta(n_0)$.
        Run the algorithm from \Cref{thm:LSV} to construct the group $\Gamma=\mathrm{PGL}_k\parens*{\F_{q^e}}$ along with generators $S=S_1\sqcup \cdots \sqcup S_{k-1}$, closed under inverses.
        \item Let $M = \Gamma$, and let $V$ be all the $(k-1)$-faces in the Cayley complex $\Cay_k(\Gamma,S)$, and define $G$ as the graph where we place an edge between $m\in M$ and $v\in V$ if $m\in v$.
        \item Identify elements in $[D]$ with $k$-sets of the form $\{\id, s_1,\dots,s_{k-1}\}$ where $s_i \in S_i$, and $s_as_b^{-1}\in S$ for all $a,b\in[k]$. 
        Here, $\id$ is the identity element of $\Gamma$.
        For each $m\in M$, we define $\Nbr_m$ as the function that maps $\{\id,s_1,\dots,s_{k-1}\}$ to $\{m,ms_1,\dots,ms_{k-1}\}$, which is a clique, and hence one of the right neighbors of $m$.
    \end{itemize}
\end{definition}

\noindent 
It remains to prove that the base graph constructed via the above procedure has useful properties.

\begin{lemma}   \label{lem:structured-bip-graphs}
    The graphs $G$ produced in \Cref{def:base-graph-construction} is a structured bipartite graph in the sense of \Cref{def:structured-bipartite-graph} with parameters $D = [k]_q!$ and $s_G(a,b) = \gbinom{k}{b-a}_q \in [q^{k-1}, O(q^{\floor{k^2/4}})]$ for $a < b \in [k]$.
\end{lemma}
\Cref{lem:structured-bip-graphs} can be proved by mechanically verifying that $G$ indeed satisfies all the claimed properties; we omit the details.

\begin{lemma}   \label{lem:triangle-bound-m}
    The graph $G$ produced in \Cref{def:base-graph-construction} is a $\tau$-small-set triangle expander for
    \[
        \tau = O_k(1) \cdot q^{\binom{k}{2} - \frac{k^2}{8} - \frac{1}{2}} \cdot \max_{0 \le i_0 < i_1 < i_2 < k} \min_{(i, j) \in \left\{ \substack{
            (i_1 - i_0, i_2 - i_0), \\
            (i_2 - i_1, k + i_0 - i_1), \\
            (k + i_0 - i_2, k + i_1 - i_2) 
        } \right\} } q^{\frac{1}{8}((i-j+2)^2 + (k-i-j)^2)} \mper
    \]
\end{lemma}

We defer the proof of \Cref{lem:triangle-bound-m} to \Cref{sec:triangles}.

\begin{lemma}   \label{lem:skeleton-expansion}
    The graph $G$ produced in \Cref{def:base-graph-construction} is a $\lambda$-small-set skeleton expander for
    \[
        \lambda \coloneqq O_k(1) \cdot q^{\frac{1}{2}\floor{\frac{k^2}{4}}}\mper
    \]
\end{lemma}
\Cref{lem:skeleton-expansion} follows from \Cref{lem:bipartite-spectral-bound} along with \Cref{lem:small-set-specrad} below.
\begin{lemma}   \label{lem:small-set-specrad}
    Let $G$ be an $n$-vertex $d$-regular graph with largest nontrivial eigenvalue $\lambda$.
    Let $H$ be any subgraph of $G$ with edges incident to at most $\eps n$ vertices.
    The largest eigenvalue of $H$ is at most $\lambda + \eps d$.
\end{lemma}
\begin{proof}
    Let $v$ be the top eigenvector of $A_H$.
    Since $A_H$ is a nonnegative matrix, by the Perron--Frobenius theorem, all entries of $v$ are nonnegative.
    Further, the entries of $v$ are $0$ on isolated vertices in $H$, which ensures that $v$ is supported on at most $\eps n$ vertices.
    We have:
    \begin{equation*}
        \lambda_{\max}(A_H) = v^{\top} A_H v \le v^{\top} A_G v = v^{\top} \parens*{A_G - \frac{d}{n}11^{\top}}v + \frac{d}{n}\angles*{v, 1}^2 \le \lambda + \eps d\mcom
    \end{equation*}
    where the second step uses nonnegativity of $v$ and $A_G-A_H$, and the final step uses Cauchy--Schwarz on $v$ and $1_{\supp(v)}$.
\end{proof}

\subsection{Imbalanced case}
In the sequel, $\Gamma$ and $S = S_1\sqcup \dots \sqcup S_{k-1}$ be as in \Cref{def:base-graph-construction}.
\begin{definition}[Face generator]
    We call $\sigma = \{s_0 = \id, s_1,\dots, s_{k-1}\}$ a \emph{face generator} if $s_i\in S_i$, and for every $i,j\in[k]$, $s_i^{-1} s_j$ is in $S$.
    Given a collection $F$ of face generators, we use $\TruncCay(\Gamma,F)$ to denote the complex obtained by including every face of the form $m\sigma$ for $m\in\Gamma$ and $\sigma\in F$.
\end{definition}

\begin{observation}
    Observe that for any collection $F$ of face generators, the set of $(k-1)$-faces in the complex $\TruncCay(\Gamma,F)$ is a subset of the set of $(k-1)$-faces in $\Cay_k(\Gamma,S)$.
\end{observation}

\begin{definition}[Equivalence relation on face generators] \label{def:equivalence-relation}
    We say face generators $\sigma_1$ and $\sigma_2$ are \emph{equivalent} if
    for some $s\in S$, we have $\sigma_1 = s^{-1} \sigma_2$.
\end{definition}

\begin{observation}
    Any equivalence class can have at most $k$ elements.
    This is because $\id$ is contained in both $\sigma_1$ and $\sigma_2$, and thus, for $\sigma_1 = s^{-1} \sigma_2$ to contain $\id$, $s$ must be in $\sigma_2$, for which there are only $k$ choices.
\end{observation}

\begin{lemma}   \label{lem:basically-arbitrary-degree}
    There exists an integer $j\in[1,k]$ such that for any prime power $q$, and integer $D < \frac{1}{k}|\bbP(\bbF_q^k)| = \Theta\left(q^{k\choose 2}\right)$ that is divisible by $j$, there is a set $F$ of face generators of size exactly $D$ such that the bipartite vertex-face incidence graph $G$ of $\TruncCay(\Gamma,F)$ is a structured $(k,D)$-biregular graph, and further, $G$ is a $\tau$-small-set triangle expander and a $\lambda$-small-set skeleton expander for $\tau$ as in \Cref{lem:triangle-bound-m} and for $\lambda$ as in \Cref{lem:skeleton-expansion}.
\end{lemma}
\begin{proof}
    Define $\calF(S)$ as the collection of all face generators on $S$.
    Observe that $\calF(S)$ is in one-to-one correspondence with faces incident to any fixed $m\in\Gamma$, and hence $|\calF(S)| = |\bbP(\bbF_q^k)|$.
    We construct $F$ be starting with $\calF(S)$ and deleting a subcollection of generators.
    First, partition $\calF(S)$ into equivalence classes based on the equivalence relation in \Cref{def:equivalence-relation}.
    For some $1\le j\le k$, at least $\frac{1}{k}|\calF(S)|$ many face generators belong to an equivalence class of size $j$.
    We construct $F$ of size $D < \frac{1}{k}|\calF(S)|$ by choosing all the face generators from $D/j$ arbitrary equivalence classes of size $j$.

    We first prove that $G$ is $(k,D)$-regular: observe that $G$ is $k$-left-regular by definition.
    To see that it is $D$-right-regular, first observe that for any $m\in \Gamma$, $m$ is contained in $D$ faces of the form $m\sigma$ for $\sigma\in F$.
    We claim that there if a face $f$ that contains $m$, then it must be of the form $m\sigma$ for some $\sigma\in F$.
    Observe that $f$ must be of the form $m'\sigma'$ for some $m'\in\Gamma$ and face generator $\sigma'$ in an equivalence class of size $j$ to be included in the complex.
    Since $m'\sigma'$ contains $m$, $m'^{-1}m$ is in $\sigma'$, and hence in $S$.
    Since $S$ is closed under inverses, $m^{-1}m'$ is also in $S$.
    Now, we choose $\sigma = m^{-1}m'\sigma'$.
    We see that $\sigma$ is a valid face generator since it contains $\id$, and for any $a,b\in\sigma$, we have
    $$
        a^{-1}b = a'^{-1} m'^{-1} m m^{-1} m' b' = a'^{-1} b'\in S.
    $$
    Additionally, since $m^{-1}m$ is in $S$, $\sigma$ and $\sigma'$ belong to the same equivalence class, which implies that $\sigma$ must be in $F$, since $\sigma'$ is in $F$.

    Finally, the claim that $G$ is a structured bipartite graph in the sense of \Cref{def:structured-bipartite-graph} is a straightforward verification that we omit, and the small-set triangle and skeleton expansion properties follow from the fact that $G$ is a subgraph of the graph constructed in \Cref{def:base-graph-construction} combined with \Cref{lem:triangle-bound-m,lem:skeleton-expansion}.
\end{proof}
\section{Bounding triangles}
\label{sec:triangles}

In this section, we prove \Cref{lem:triangle-bound-m}.
Fix a prime power $q$ and integers $k, e \ge 2$. Let $X$ denote the resulting Ramanujan clique complex as constructed in \Cref{thm:LSV}. Recall that we denote the vertex set of $X$ by $X(0) = V_0 \cup V_1 \cup \dots \cup V_{k-1}$, where $V_i$ is the vertices in the $i$'th part. All $V_i$'s have the same size, which we denote in this section by $n$. Let $d_{ij}$ denote the degree of a left vertex in the bipartite graph $(V_i, V_j)$ induced by the $1$-skeleton of $X$. Recall that $d_{ij} = \gbinom{k}{(j-i)_k}_q \approx q^{(j-i)_k(k-(j-i)_k)}$, where in general we use $a_k \in \{ 0, 1, \dots, k-1 \}$ to denote the number $a$ reduced modulo $k$.

Our goal in this section is to upper bound the number of triangles contained in any small set $U \subseteq X(0)$. Our bounds follow from spectral bounds on edge density in the Ramanujan graph and in its links.
The proof strategy loosely resembles the one taken in the proof of \cite[Theorem 10.14]{DH24}.

\subsection{Edge density of small sets in the Ramanujan complex}

We will first prove some statements related to the edge density in small sets.

\begin{lemma} \label{lem:U-ia-density}
    Let $\delta < \frac{1}{q^{k^2/4}}$ and $\alpha \in \bbN$. Let $U_{i,\alpha} \subseteq V_i$ and $U_j \subseteq V_j$ be such that $|U_{i,\alpha}| \le \delta n$ and $|U_j| \le \delta n$. Suppose also that the number of neighbors each vertex $u \in U_{i,\alpha}$ has in $U_j$ is in $[2^{\alpha-1}, 2^\alpha)$. 
    Then,
    \[
        |U_{i,\alpha}| \le O_k(1) \cdot q^{(j-i)_k (k-(j-i)_k)} \cdot 2^{-2\alpha} \cdot |U_j|.
    \]
\end{lemma}
We use the bipartite expander mixing lemma to prove the above statement.
\begin{lemma}[Bipartite Expander-Mixing Lemma] \label{lem:EML}
    Let $G = (L,R,E)$ be a bipartite graph with left-degree~$c$ and right-degree~$d$.
    Let $A \subseteq L$ have $|A|/|L| = \alpha$ and $B \subseteq R$ have $|B|/|R| = \beta$.
    Let $\lambda$ denote the magnitude of the largest nontrivial (not~$\pm \sqrt{cd}$) eigenvalue of $G$'s adjacency matrix. 
    Then
    \begin{equation}
        \abs*{\frac{|E(A,B)|}{|E(L,R)|} - \alpha \beta} 
        \leq \frac{\lambda}{\sqrt{cd}}\sqrt{\alpha(1-\alpha) \beta(1-\beta)}
        \leq \frac{\lambda}{\sqrt{cd}}\sqrt{\alpha \beta} \mper
    \end{equation}
\end{lemma}

\begin{proof}[Proof of \Cref{lem:U-ia-density}]
    Consider the bipartite graph $(V_i, V_j)$ induced by the $1$-skeleton of $X$. Notice that both the left and right degrees are $d_{ij} = \gbinom{k}{(j-i)_k}_q$. By \Cref{lem:bipartite-spectral-bound}, the second eigenvalue $\lambda_2(V_i, V_j)$ of this graph is bounded above by $O_k(1) \cdot \sqrt{q^{(j-i)_k(k-(j-i)_k)}}$. 
    
    We know that $e(U_{i,\alpha}, U_j) \ge 2^{\alpha-1} \cdot |U_{i,\alpha}|$. By the bipartite expander mixing lemma (\Cref{lem:EML}), we have that 
    \begin{align*}
        e(U_{i,\alpha}, U_j) 
        &\le \frac{|U_{i,\alpha}| |U_j| d_{ij}}{n} + \lambda_2(V_i, V_j) \cdot \sqrt{|U_{i,\alpha}| |U_j|} \\
        &\le \left( \delta d_{ij} + \lambda_2(V_i, V_j) \right) \cdot \sqrt{|U_{i,\alpha}| |U_j|} \\
        &= O_{k}(1) \cdot \sqrt{q^{(j-i)_k(k-(j-i)_k)}} \cdot \sqrt{|U_{i,\alpha}| |U_j|}.
    \end{align*}
    Combining inequalities, this tells us that
    \[
        2^{\alpha - 1} \cdot |U_{i,\alpha}|
        \le e(U_{i,\alpha}, U_j) 
        \le O_{k}(1) \cdot \sqrt{q^{(j-i)_k(k-(j-i)_k)}} \cdot \sqrt{|U_{i,\alpha}| |U_j|},
    \]
    which implies that 
    \[
        |U_{i,\alpha}| \le O_k(1) \cdot q^{(j-i)_k (k-(j-i)_k)} \cdot 2^{-2\alpha} \cdot |U_j| \mper
        \qedhere
    \]
\end{proof}

The next lemma gives a bound on edge density between parts within the link of a vertex. For $u \in X(0)$, let $V_j(u)$ denote all the vertices in $V_j$ that share an edge with $u$. 

\begin{lemma} \label{lem:link-edge-density}
    For $0 < i < j \in [k]$, and for $u \in V_0$, $U_i(u) \subseteq V_i(u)$, and $U_j(u) \subseteq V_j(u)$, it holds that 
    \[
        e(U_i(u), U_j(u)) \le O_k(1) \cdot \left( \frac{|U_i(u)| |U_j(u)|}{q^{i(k-j)}} +  \sqrt{q^{(j-i)(k-(j-i)-1)}} \cdot \sqrt{|U_i(u)| |U_j(u)|} \right).
    \]
\end{lemma}

\begin{proof}
    Let $L$ be the bipartite graph in the link on $(V_i(u), V_j(u))$.
    The left side has $d_{0i} = \gbinom{k}{i}_q \approx q^{i(k-i)}$ vertices, and the right side has $d_{0j} = \gbinom{k}{j}_q \approx q^{j(k-i)}$ vertices.
    The left degree is $\gbinom{k-i}{k-j}_q \approx q^{(j-i)(k-j)}$ and the right degree is $\gbinom{j}{i}_q \approx q^{i(j-i)}$.

    By \Cref{lem:link-second-eigenvalue}, the second eigenvalue of this graph is at most $\sqrt{q^{(j-i)(k-(j-i)-1)}}$.
    Thus, by the bipartite expander mixing lemma (\Cref{lem:EML}), we have:
    \begin{align*}
        e(U_i(u), U_j(u)) 
        &\le \frac{|U_i(u)| |U_j(u)| \sqrt{\gbinom{k-i}{k-j}_q \gbinom{j}{i}_q}}{\sqrt{\gbinom{k}{i}_q \gbinom{k}{j}_q}} + \lambda_2(V_i(u), V_j(u)) \cdot \sqrt{|U_i(u)| |U_j(u)|} \\
        &= O_k(1) \cdot \left( \frac{|U_i(u)| |U_j(u)|}{q^{i(k-j)}} +  \sqrt{q^{(j-i)(k-(j-i)-1)}} \cdot \sqrt{|U_i(u)| |U_j(u)|} \right).  \qedhere
    \end{align*}
\end{proof}

\subsection{Warm up: bounding triangles in small sets in the 3-partite Ramanujan complex}

Recall that our main goal of this section is to bound the number of triangles contained within small sets of vertices in the $k$-partite Ramanujan complex. As a warm-up, here we will calculate this quantity for the case of $k=3$. The same ideas will generalize to the general case in \Cref{sec:triangles-in-U-general-k}, but we find it more transparent to first demonstrate this easier calculation.

In the case of $k=3$, the vertices of the Ramanujan complex is of the form $X(0) = (V_0, V_1, V_2)$. Let $U = U_0 \cup U_1 \cup U_2 \subseteq X(0)$, where $U_i \subseteq V_i$ for $i \in [3]$, be a small set of vertices of $X$, satisfying that $|U| \le \delta \cdot n = \frac{\delta}{3} \cdot |X(0)|$ where $\delta = \frac1{q^{9/4}}$ as to satisfy the conditions of~\Cref{lem:U-ia-density}. 

The Ramanujan complex is a partite complex, with edges only between different parts, and hence any triangle contained within $U$ must have exactly one vertex from each of $U_0$, $U_1$, and $U_2$.
In fact, if we fix $u \in U_0$ and denote $U_1(u) := U_1 \cap V_1(u)$ and $U_2(u) := U_2 \cap V_2(u)$, then the number of triangles containing $u$ is precisely the number of edges within the graph $(U_1(u), U_2(u))$. 
This latter quantity can be bounded by an application of the expander mixing lemma to the link of $u$ (formally given in \Cref{lem:link-edge-density}). 

Our strategy for bounding the number of triangles contained within $U$ can be described as follows:
\begin{enumerate}[{(1)}]
\item 
    We will split the vertices in $U_0$ according to the number of neighbors it has within $U_1 \cup U_2$. For any $\alpha > 0$, \Cref{lem:U-ia-density} gives us an upper bound on the number of vertices $u \in U_0$ that can have $\approx 2^\alpha$ neighbors in $U_1 \cup U_2$. 
\item 
    If a vertex $u \in U_0$ has $\approx 2^\alpha$ neighbors in $U_1 \cup U_2$, then the number of triangles containing $u$ is equal to the number of edges in $(U_1(u), U_2(u))$ within the link of $u$. \Cref{lem:link-edge-density} gives us an upper bound on this quantity in terms of $\alpha$.
\end{enumerate}

We now execute this strategy. Let $\triangle$ denote the number of triangles contained within $U$.

\begin{lemma} \label{lem:k=3-triangle-bound}
    Let $\delta < \frac1{q^{9/4}}$. For any $U = U_0 \cup U_1 \cup U_2 \subseteq X(0)$ with $|U| \leq \delta |X(0)|$, the number of triangles in $X(2)$ contained in $U$ is at most $O(q^{3/2})|U|$.
\end{lemma}

\begin{proof}
    First, notice that for any $u \in U_0$ the number of neighbors it has within $V_1 \cup V_2$, and hence the maximal possible number of neighbors it has within $U_1 \cup U_2$, is upper bounded by $d_{01} + d_{02} = O(q^2)$. Let $B = \lceil \log_2 (d_{01} + d_{02}) \rceil + 1 = O(\log_2 q)$. 
    
    We partition $U_0$ into sets $U_{0,\emptyset} \cup U_{0,1} \cup U_{0,2} \cup \dots \cup U_{U_0,B}$ where 
    \[ 
        U_{0,\alpha} = \{ u \in U_0 : |U_1(u)| + |U_2(u)| \in [2^{\alpha-1}, 2^\alpha) \}
    \]
    and $U_{0,\emptyset}$ denotes all vertices in $U_0$ with no neighbors in $U_1 \cup U_2$. The vertices in $U_{0,\emptyset}$ contribute no triangles to our count, so we may ignore them. For any $\alpha \in [1, B]$, combining with the trivial observation that $|U_{0,\alpha}| \le |U|$, \Cref{lem:U-ia-density} gives us an upper bound on the size of $U_{0,\alpha}$:
    \[
        |U_{0,\alpha}| \le \min \braces*{ 1, O(q^2 2^{-2\alpha}) } \cdot |U| \mper \numberthis \label{eqn:U_0,a}
    \]

    The number of triangles contained in $U$ is equal to
    $\triangle = \sum_{u\in U_0} e(U_1(u), U_2(u))$.
    By \Cref{lem:link-edge-density}, for any $\alpha$ and $u\in U_{0,\alpha}$, we have
    \[
        e(U_1(u), U_2(u)) \le O(1) \cdot \left( q^{-1} 2^{2\alpha} + \sqrt{q} \cdot 2^\alpha \right).
    \]
    Thus, we may upper bound the number of triangles in $U$ by
    \begin{align*}
        \triangle \le O(1) \cdot \sum_{\alpha=1}^B |U_{0,\alpha}| \cdot \left( q^{-1} 2^{2\alpha} + \sqrt{q} \cdot 2^\alpha \right). 
    \end{align*}
    For $2^\alpha \le q$, we have that $|U_{0,\alpha}| \le |U|$. For $2^\alpha > q$, Equation~\eqref{eqn:U_0,a} tells us that $|U_{0,\alpha}| \le O(q^2 2^{-2\alpha}) |U|$, in which case the summand $|U_{0,\alpha}| \cdot \left( q^{-1} 2^{2\alpha} + \sqrt{q} \cdot 2^\alpha \right) \le q + q^{5/2} 2^{-\alpha}$. Combining the two, we get that the number of triangles in $U$ is bounded by:
    \begin{align*}
        \triangle 
        &\le \sum_{\alpha : 2^\alpha \le q} \left( q^{-1} 2^{2\alpha} + \sqrt{q} \cdot 2^\alpha \right) \cdot |U| + \sum_{\alpha : 2^\alpha > q} \left( q + q^{5/2} 2^{-\alpha} \right) \\
        &= O(1) \cdot |U| \cdot \left( q^{-1} \cdot q^2 + \sqrt{q} \cdot q + q \log q + q^{5/2} \cdot q^{-1} \right) \\
        &\le O(q^{3/2}) |U|,
    \end{align*}
    completing the proof.
\end{proof}

\subsection{Bounding triangles in small sets: general case} \label{sec:triangles-in-U-general-k}

We now bound the number of triangles within small sets of the $k$-partite Ramanujan complex.
Our strategy is similar to the $k=3$ case considered in the previous section but is more notationally complex. 

Let $U = U_0 \cup U_1 \cup \dots \cup U_{k-1} \subseteq X(0)$ be a small set of size $\le \delta n = \frac{\delta}{k} \cdot |X(0)|$.
Our goal is to bound the number of triangles contained within $U$.

Fix three parts $i_0 < i_1 < i_2 \in [k]$.
Our goal is to bound the number of triangles contained within $U_{i_0} \cup U_{i_1} \cup U_{i_2}$ in terms of $|U|$, which we will denote by $\triangle_{i_0, i_1, i_2}$.
In fact, by rotational symmetry of the Ramanujan complex, we may assume that $i_0 = 0$ (by replacing $i_1$ with $i_1 - i_0$ and $i_2$ with $i_2 - i_0$). So, let us relabel and consider the case of triangles within $U_0 \cup U_i \cup U_j$, where $0 < i < j < k$. 

\begin{lemma}
    Let $\delta < \frac{1}{q^{k^2/4}}$. Let $U = U_0 \cup U_1 \cup \dots \cup U_{k-1} \subseteq X(0)$ such that $|U| \le \delta n$. For any $0 < i < j < k$, it holds that
    \[
        \triangle_{0,i,j} \le O_k(1) \cdot \sqrt{q^{(j-i)(k-(j-i)-1)}} \cdot q^{i(k-i)/4} \cdot q^{j(k-j)/4} \cdot |U| \mper
    \]
\end{lemma}

\begin{proof}
    For any $u \in U_0$, note that $|V_i(u)| \le d_{0i} < q^{k^2/4}$ and $|V_j(u)| \le d_{0j} < q^{k^2/4}$. Let $B = \lceil \log_2 q^{k^2/4} \rceil = O_k(1) \cdot \log_2 q$.
    
    Define $U_i(u) := U_i \cap V_i(u)$. We will partition $U_0$ into sets $U_{0,\emptyset} \cup (\bigcup_{0 < \alpha, \beta \le B} U_{0, \alpha, \beta})$, where $U_{0,\alpha,\beta} \subseteq U_0$ denotes the set of vertices $u \in U_0$ with $|U_i(u)| \in [2^{\alpha-1}, 2^\alpha)$ and $|U_j(u)| \in [2^{\beta-1}, 2^\beta)$, and $U_{0,\emptyset} \subseteq U_0$ denotes the set of vertices $u \in U_0$ with either $U_i(u) = \emptyset$ or $U_j(u) = \emptyset$. Notice that none of the $u \in U_{0,\emptyset}$ contribute any triangles.

    By \Cref{lem:U-ia-density} and also observing that $|U_{0,\alpha,\beta}| \le |U_0| \le |U|$, we have that 
    \[
        |U_{0,\alpha,\beta}| \le O_k(1) \cdot \min \left\{ 1, q^{i(k-i)} \cdot 2^{-2\alpha}, q^{j(k-j)} \cdot 2^{-2\beta} \right\} \cdot |U|.
    \]
    For any $u \in U_{0,\alpha,\beta}$, we know that $|U_i(u)| \in [2^{\alpha-1}, 2^\alpha)$ and $|U_j(u)| \in [2^{\beta-1}, 2^\beta)$. Thus, by \Cref{lem:link-edge-density} we have:
    \[
        e(U_i(u), U_j(u)) \le O_k(1) \cdot \left( \frac{2^{\alpha + \beta}}{q^{i(k-j)}} + \sqrt{q^{(j-i)(k-(j-i)-1)}} \cdot 2^{(\alpha + \beta)/2} \right).
    \]
    Altogether, this gives us that 
    \begin{align*}
        \triangle_{0,i,j} 
        &= \sum_{u \in U_0} e(U_i(u), U_j(u)) \\
        &\le \sum_{0 < \alpha, \beta \le B} |U_{0,\alpha,\beta}| \cdot O_k(1) \cdot \left( \frac{2^{\alpha + \beta}}{q^{i(k-j)}} + \sqrt{q^{(j-i)(k-(j-i)-1)}} \cdot 2^{(\alpha + \beta)/2} \right) \\
        &\le O_k(|U|) \cdot \sum_{0 < \alpha, \beta \le B} \min \{ 1, q^{i(k-i)} \cdot 2^{-2\alpha}, q^{j(k-j)} \cdot 2^{-2\beta} \} \cdot \left( \frac{2^{\alpha + \beta}}{q^{i(k-j)}} + \sqrt{q^{(j-i)(k-(j-i)-1)}} \cdot 2^{(\alpha + \beta)/2} \right) \\
        &= O_k(|U|) \cdot \left( \mathsf{term1} + \mathsf{term2} + \mathsf{term3} \right) \mcom
    \end{align*}
    where 
    \begin{align*}
        \mathsf{term1} &= \sum_{\substack{
            0 < \alpha \le \frac{i(k-i)}{2} \log_2 q \\
            0 < \beta \le \frac{j(k-j)}{2} \log_2 q 
        }} \left( \frac{2^{\alpha + \beta}}{q^{i(k-j)}} + \sqrt{q^{(j-i)(k-(j-i)-1)}} \cdot 2^{(\alpha + \beta)/2} \right) \\
        \mathsf{term2} &= \sum_{\substack{
            \frac{i(k-i)}{2} \log_2 q < \alpha \le B \\
            0 < \beta \le \alpha + \frac{j(k-j) - i(k-i)}{2} \log_2 q
        }} q^{i(k-i)} \cdot 2^{-2\alpha} \cdot \left( \frac{2^{\alpha + \beta}}{q^{i(k-j)}} + \sqrt{q^{(j-i)(k-(j-i)-1)}} \cdot 2^{(\alpha + \beta)/2} \right) \\
        \mathsf{term3} &= \sum_{\substack{
            \frac{j(k-j)}{2} \log_2 q < \beta \le B \\
            0 < \alpha \le \beta + \frac{i(k-i) - j(k-j)}{2} \log_2 q
        }} q^{j(k-j)} \cdot 2^{-2\beta} \cdot \left( \frac{2^{\alpha + \beta}}{q^{i(k-j)}} + \sqrt{q^{(j-i)(k-(j-i)-1)}} \cdot 2^{(\alpha + \beta)/2} \right).
    \end{align*}
    We bound each term separately. For $\mathsf{term1}$, by separately evaluating the sum over $\beta$, then $\alpha$, we have that
    \begin{align*}
        \mathsf{term1} 
        &= \sum_{\substack{
            0 < \alpha \le \frac{i(k-i)}{2} \log_2 q \\
            0 < \beta \le \frac{j(k-j)}{2} \log_2 q 
        }} \left( \frac{2^{\alpha + \beta}}{q^{i(k-j)}} + \sqrt{q^{(j-i)(k-(j-i)-1)}} \cdot 2^{(\alpha + \beta)/2} \right) \\
        &= \sum_{0 < \alpha \le \frac{i(k-i)}{2} \log_2 q} O(1) \cdot \left( \frac{2^{\alpha} \cdot q^{j(k-j)/2}}{q^{i(k-j)}} + \sqrt{q^{(j-i)(k-(j-i)-1)}} \cdot 2^{\alpha/2} \cdot 2^{j(k-j)/4} \right) \\
        &= O(1) \cdot \left( \frac{q^{i(k-i)/2} \cdot q^{j(k-j)/2}}{q^{i(k-j)}} + \sqrt{q^{(j-i)(k-(j-i)-1)}} \cdot q^{i(k-i)/4} \cdot q^{j(k-j)/4} \right) \\
        &= O(1) \cdot \left( q^{(j-i)(k-(j-i))/2} + \sqrt{q^{(j-i)(k-(j-i)-1)}} \cdot q^{i(k-i)/4} \cdot q^{j(k-j)/4} \right).
    \end{align*}
    Next, for $\mathsf{term2}$, we first sum over $\beta$, then over $\alpha$, getting
    \begin{align*}
        \mathsf{term2} 
        &= \sum_{\substack{
            \frac{i(k-i)}{2} \log_2 q < \alpha \le B \\
            0 < \beta \le \alpha + \frac{j(k-j) - i(k-i)}{2} \log_2 q
        }} q^{i(k-i)} \cdot 2^{-2\alpha} \cdot \left( \frac{2^{\alpha + \beta}}{q^{i(k-j)}} + \sqrt{q^{(j-i)(k-(j-i)-1)}} \cdot 2^{(\alpha + \beta)/2} \right) \\
        &= O(1) \cdot \sum_{\frac{i(k-i)}{2} \log_2 q < \alpha < B} q^{i(k-i)} \cdot 2^{-2\alpha} \cdot \left( \frac{2^{2\alpha} \cdot q^{(j(k-j) - i(k-i))/2}}{q^{i(k-j)}} + \sqrt{q^{(j-i)(k-(j-i)-1)}} \cdot 2^\alpha \cdot q^{(j(k-j) - i(k-i))/4} \right) \\
        &= O(1) \cdot \sum_{\frac{i(k-i)}{2} \log_2 q < \alpha < B} \left( q^{(j-i)(k-(j-i))/2} + \sqrt{q^{(j-i)(k-(j-i)-1)}} \cdot 2^{-\alpha} \cdot q^{(j(k-j) + 3i(k-i))/4} \right) \\
        &= O_k(1) \cdot \left( q^{(j-i)(k-(j-i))/2} \cdot \log_2 q + \sqrt{q^{(j-i)(k-(j-i)-1)}} \cdot q^{i(k-i))/4} \cdot q^{(j(k-j)/4} \right).
    \end{align*}
    Similarly, for $\mathsf{term3}$, we first sum over $\alpha$ then $\beta$, and get
    \[
        \mathsf{term3} \le O_k(1) \cdot \left( q^{(j-i)(k-(j-i))/2} \cdot \log_2 q + \sqrt{q^{(j-i)(k-(j-i)-1)}} \cdot q^{i(k-i)/4} \cdot q^{j(k-j)/4} \right).
    \]
    
    Now, notice that all three terms are of the form 
    \[
        O_k(1) \cdot \left( q^{(j-i)(k-(j-i))/2} \cdot \log_2 q + \sqrt{q^{(j-i)(k-(j-i)-1)}} \cdot q^{i(k-i)/4} \cdot q^{j(k-j)/4} \right).
    \]
    In fact, it always holds that 
    \[
        q^{(j-i)(k-(j-i))/2} \cdot \log_2 q = o\left(  \sqrt{q^{(j-i)(k-(j-i)-1)}} \cdot q^{i(k-i)/4} \cdot q^{j(k-j)/4} \right), \numberthis \label{eq:eml-compare-terms}
    \]
    so all three terms are 
    \[
        O_k(1) \cdot \sqrt{q^{(j-i)(k-(j-i)-1)}} \cdot q^{i(k-i)/4} \cdot q^{j(k-j)/4} \mcom
    \]
    so
    \[
        \triangle_{0,i,j} \le O_k(1) \cdot \sqrt{q^{(j-i)(k-(j-i)-1)}} \cdot q^{i(k-i)/4} \cdot q^{j(k-j)/4} \cdot |U| \mper
    \]
    Finally, to see~\eqref{eq:eml-compare-terms}, we divide both sides by $q^{(j-i)(k-(j-i))/2}$, which gives
    \begin{align*}
        \log_2 q \leq o(1) \cdot q^{-(j-i)/2} \cdot q^{i(k-i)/4} \cdot q^{j(k-j)/4} = o(1) \cdot q^{(i(k-i+2) + j(k-j-2))/4} \mper
    \end{align*}
    Since $1 \leq i < j \leq k-1$, it follows that $i(k-i+2)$ is minimized at $i=1$, and $j(k-j-2) < 0$ only if $j = k-1$.
    \begin{align*}
        \frac{1}{4} \parens*{i(k-i+2) + j(k-j-2)}
        \geq \frac{1}{4} ((k+1) - (k-1)) = \frac{1}{2} \mper
    \end{align*}
    This implies that
    \begin{align*}
        \frac{(j-i)(k-(j-i))}{2} + \frac{1}{2}
        \leq \frac{(j-i)(k-(j-i)-1)}{2} + \frac{i(k-i)}{4} + \frac{j(k-j)}{4} \mcom
    \end{align*}
    thus establishing \eqref{eq:eml-compare-terms}.
    % \begin{align*}
    %     2(j-i) + 1 
    %     &< (j-i)(i+1) + 2i \\
    %     &= (i+j)(j+1) - i^2 - j^2 \\
    %     \implies 2(j-i) + 1 
    %     &< (i+j)k - i^2 - j^2 \\
    %     &= i(k-i) + j(k-j) \\
    %     \implies \frac{(j-i)(k-(j-i))}{2} + 1 
    %     &< \frac{(j-i)(k-(j-i)-1)}{2} + \frac{i(k-i)}{4} + \frac{j(k-j)}{4} \\
    %     \implies q^{(j-i)(k-(j-i))/2} \cdot \log_2 q 
    %     &\le o(1) \cdot \sqrt{q^{(j-i)(k-(j-i)-1)}} \cdot q^{i(k-i)/4} \cdot q^{j(k-j)/4} \mper \qedhere
    % \end{align*}
\end{proof}

It follows by the rotational symmetry of the Ramanujan complex that a similar bound holds for any choice of three parts.

\begin{corollary} \label{cor:triangle-bound}
    Let $\delta < \frac{1}{q^{k^2/4}}$. Let $U = U_0 \cup U_1 \cup \dots \cup U_{k-1} \subseteq X(0)$ such that $|U| \le \delta n$. For any $i_0 < i_1 < i_2 \in [k]$, it holds that
    \[
        \triangle_{i_0,i_1,i_2} \le O_k(1) \cdot \min_{(i,j) \in \left\{ 
        \substack{
            (i_1 - i_0, i_2 - i_0), \\ (i_2 - i_1, k + i_0 - i_1), \\(k + i_0 - i_2, k + i_1 - i_2)
        }
        \right\}} 
        \left[ \sqrt{q^{(j-i)(k-(j-i)-1)}} \cdot q^{i(k-i)/4} \cdot q^{j(k-j)/4} \cdot |U| \right].
    \]
\end{corollary}

\subsection{Bounding faces with triangles in small sets}

In this subsection, we bound the number of $(k-1)$-faces of $X$ that contain a triangle within a small vertex set $U$. We let this set of $(k-1)$-faces be denoted by $F^{k,3}(U)$. Formally, we define:
\[
    F^{k,3}(U) = \left\{ f \in X(k-1) : \left| f \cap U \right| \ge 3 \right\}.
\]
To bound $|F^{k,3}(U)|$, we will bound for each $i_0 \not= i_1 \not= i_2 \in [k]$ the number $\Delta_{i_0,i_1,i_2}$ of triangles contained within $U_{i_0} \cup U_{i_1} \cup U_{i_2}$, then multiply by the number of ways to extend each triangle to a $(k-1)$-face. We've already given an upper bound on $\Delta_{i_0,i_1,i_2}$ in \Cref{cor:triangle-bound}. The second quantity is given in the following claim. 

\begin{claim} \label{claim:k-faces-per-triangle}
    For $0 \le i_0 < i_1 < i_2 < k$, and for $v_{i_0} \in V_{i_0}$, $v_{i_1} \in V_{i_1}$, and $v_{i_2} \in V_{i_2}$ such that $\{ v_{i_0}, v_{i_1}, v_{i_2} \} \in X(2)$, there are $O_k(1) \cdot q^{\binom{i_1-i_0}{2} + \binom{i_2-i_1}{2} + \binom{k+i_0-i_2}{2}}$ $(k-1)$-faces containing $\{ v_{i_0}, v_{i_1}, v_{i_2} \}$.
\end{claim}

\begin{proof}
    Let us set $i_1' = i_1 - i_0$ and $i_2' = i_2 - i_0$, and look within the link of $v_{i_0}$. The vertices $v_{i_1}$ and $v_{i_2}$ correspond to a $i_1'$-plane $\rho_{i_1'}$ and $i_2'$-plane $\rho_{i_2'}$ within $\bbP\parens*{\bbF_q^k}$.
    The quantity we are interested in is the number of ways to choose a sequence of planes $\rho_1 \subseteq \rho_2 \subseteq \dots \subseteq \rho_{k-1}$ with $\rho_{i_1'}$ and $\rho_{i_2'}$ fixed, which is equal to 
    \begin{align*}
        \left( \gbinom{i_1'}{1}_q \cdot \frac{\gbinom{i_1'}{2}_q}{\gbinom{2}{1}q} \cdot \dots \cdot \frac{\gbinom{i_1'}{i_1'-1}_q}{\gbinom{i_1'-1}{i_1'-2}_q} \right) 
        \cdot \left( \frac{\gbinom{i_2'}{i_1'+1}_q}{\gbinom{i_1'+1}{i_1'}_q} \cdot \dots \cdot \frac{\gbinom{i_2'}{i_2'-1}_q}{\gbinom{i_2'-1}{i_2'-2}_q} \right) 
        \cdot \left( \frac{\gbinom{k}{i_2'+1}_q}{\gbinom{i_2'+1}{i_2'}_q} \cdot \dots \cdot \frac{\gbinom{k}{k-1}_q}{\gbinom{k-1}{k-2}_q} \right),
    \end{align*}
    which is $O_k(1) \cdot q^{\binom{i_1'}{2} + \binom{i_2'-i_1'}{2} + \binom{k-i_2'}{2}} = O_k(1) \cdot q^{\binom{i_1-i_0}{2} + \binom{i_2-i_1}{2} + \binom{k+i_0-i_2}{2}}$.
\end{proof}

Combining \Cref{cor:triangle-bound} and \Cref{claim:k-faces-per-triangle}, we obtain the following bound on $F^{k,3}(U)$.

\begin{lemma} \label{lem:Fk3-bound}
    Let $\delta < \frac{1}{q^{k^2/4}}$. Let $U = U_0 \cup U_1 \cup \dots \cup U_{k-1} \subseteq X(0)$ such that $|U| \le \delta n$. Then, 
    \[
        |F^{k,3}(U)| \le O_k(1) \cdot q^{\binom{k}{2} - \frac{k^2}{8} - \frac{1}{2}} \cdot \max_{0 \le i_0 < i_1 < i_2 < k} \min_{(i, j) \in \left\{ \substack{
            (i_1 - i_0, i_2 - i_0), \\
            (i_2 - i_1, k + i_0 - i_1), \\
            (k + i_0 - i_2, k + i_1 - i_2) 
        } \right\} } q^{\frac{1}{8}((i-j+2)^2 + (k-i-j)^2)} \cdot |U| \mper
    \]
\end{lemma}
\begin{proof}
    Combining \Cref{cor:triangle-bound} and \Cref{claim:k-faces-per-triangle}, we get an upper bound of
    \begin{align*}
        O_k(1) \cdot 
        \sqrt{q^{(j-i)(k-(j-i)-1)}} \cdot q^{i(k-i)/4} \cdot q^{j(k-j)/4} \cdot q^{\binom{i}{2} + \binom{j-i}{2} + \binom{k-j}{2}} \cdot |U| \mcom
    \end{align*}
    where $i < j$ are the indices obtained from taking the maximum over $i_0 < i_1 < i_2 < k$ in \Cref{claim:k-faces-per-triangle} and the minimum as in \Cref{cor:triangle-bound}.
    It is a straightforward calculation that this simplifies to the expression in the lemma statement.
\end{proof}

\subsection{The case of \texorpdfstring{$k = 5$}{k = 5}} 

Finally, we apply our bounds from this section to the specific case of $k = 5$, which we will use in our construction of two-sided unique neighbor expanders. 

\begin{corollary} \label{cor:F53-bound}
    Let $k = 5$, and let $\delta < q^{-25/4}$. Let $U = U_0 \cup U_1 \cup U_2 \cup U_3 \cup U_4 \subseteq X(0)$ such that $|U| \le \delta n$. Then, 
    \[
        |F^{5,3}(U)| \le O(q^{13/2}  \cdot |U| ) \mper
    \]
\end{corollary}

\begin{proof}
    By \Cref{lem:Fk3-bound}, we would like to evaluate the maximum over all tuples $0 \le i_0 < i_1 < i_2 < 5$ of
    \[
        \min_{(i, j) \in \left\{ \substack{
            (i_1 - i_0, i_2 - i_0), \\
            (i_2 - i_1, k + i_0 - i_1), \\
            (k + i_0 - i_2, k + i_1 - i_2) 
        } \right\} } O_k(1) \cdot q^{\binom{k}{2} - \frac{k^2}{8} - \frac{1}{2}} \cdot q^{\frac{1}{8}((i-j+2)^2 + (k-i-j)^2)} \cdot |U| \mper \numberthis \label{eqn:F53-bound}
    \]
    Let us perform casework on the value of $(i_0, i_1, i_2)$. Because the Ramanujan complex is rotationally symmetric, there are essentially two cases to consider: the case where $i_0, i_1, i_2$ are consecutive indices, or the case where they are not all consecutive. The first case is equivalent to the case of $0,1,4$, and the second case is equivalent to the case of $0,2,3$.

    In the first case, where $i_0 = 0$, $i_1 = 1$, and $i_2 = 4$, we have that Equation~\eqref{eqn:F53-bound} is at most
    \[
        O(1) \cdot q^{3/2} \cdot q^{4/4} \cdot q^{4/4} \cdot q^{3} \cdot |U| = O(1) \cdot q^{13/2} \cdot |U| \mper
    \]
    In the second case, where $i_0 = 0$, $i_1 = 2$, and $i_2 = 3$, we have that Equation~\eqref{eqn:F53-bound} is at most 
    \[
        O(1) \cdot q^{3/2} \cdot q^{3/2} \cdot q^{3/2} \cdot q^2 \cdot |U| = O(1) \cdot q^{13/2} \cdot |U| \mper \qedhere
    \]
\end{proof}

\section{Random gadget analysis}
\label{sec:random-gadget}
In this section, we prove \Cref{lem:pseudorandom-gadget-exists} which states that there exist bipartite graphs $H$ such that $H$ and $H^{\top}$ both satisfy the properties in \Cref{def:pr-gadget}.

\restatedefinition{def:pr-gadget}

In fact, we will prove that a random one satisfies the properties with high probability.
The desired statement then follows since $H$ and $H^{\top}$ have the same distribution.
Throughout this section, we will write random variables in \textbf{boldface}.

\begin{lemma}[Consequence of Lemma B.1, B.2 of \cite{HMMP24}]  \label{lem:erdos-renyi-gadget}
    Let $\bH$ be a random $(d_1,d_2)$-biregular graph on vertex sets $A \cup B$, and let $n_1 = |A|$, $n_2=|B|$, $n = n_1 + n_2$, and $p = d_2/n_1$ (observe that $d_2/n_1 = d_1/n_2$).
    Suppose $p \leq o_n(1)$ and $d_1, d_2 \geq \log^2 n$.
    Then with probability $1-o_n(1)$, $\bH$ satisfies the following properties:
    \begin{itemize}
        \item For every non-empty $S\subseteq A$, $\frac{|\UN_{\bH}(S)|}{|S|} \geq d_1 \parens*{1-p}^{|S|-1} - \sqrt{4 p \parens*{1-p}^{|S|-1} n_1 \log n_1} - o_n(d_1)$.
        \item For every non-empty $S\subseteq B$, $\frac{|\UN_{\bH}(S)|}{|S|} \geq d_2 \parens*{1-p}^{|S|-1} - \sqrt{4 p \parens*{1-p}^{|S|-1} n_2 \log n_2} - o_n(d_2)$.
    \end{itemize}
\end{lemma}

Note that \Cref{lem:erdos-renyi-gadget} implies that a random biregular graph satisfies property (\ref{property:lossless}) in \Cref{def:pr-gadget}.
We now proceed to prove property (\ref{property:bucket-spread}).

We first state the following standard concentration bound, which was also used in \cite{HMMP24}.

\begin{lemma}[Concentration for sampling without replacement] \label{lem:without-replacement}
    Fix $1 \leq \ell\leq n$.
    Let $S\subseteq [n]$, let $\bT$ be a random sample of $\ell$ elements from $[n]$ without replacement, and let $\mu = \frac{\ell}{n}|S|$.
    Then, for all $\delta > 0$,
    \begin{align*}
        &\Pr\bracks*{ |S \cap \bT| \geq (1+\delta) \mu}
        \leq \exp\parens*{-\frac{\delta^2 \mu}{2+\delta}} \mper
    \end{align*}
\end{lemma}

We now prove that a random biregular graph satisfies property (\ref{property:bucket-spread}) of \Cref{def:pr-gadget}.

\begin{lemma} \label{lem:ER-bucketing}
    Let $\bH$ be a random $(d_1,d_2)$-biregular graph on vertex sets $A \cup B$, and let $n_1 = |A|$, $n_2=|B|$ and $n = n_1 + n_2$.
    Let $B = B_1 \cup \cdots \cup B_r$ be a fixed partition of $B$ such that $\frac{n_2}{2r} \leq |B_i| \leq \frac{2n_2}{r}$ for each $i\in [r]$.
    Then, with probability $1 - O(1/n)$, we have that for all $S\subseteq A$ with $|S| \leq n_2/d_1$ and all $W\subseteq [r]$ with $|W| \geq \frac{r\log n}{d_1}$,
    \begin{align*}
        \sum_{i\in W} |N(S) \cap B_i| \leq 32|W| \cdot \max \braces*{ \frac{d_1}{r}|S|,\ \log n} \mper
    \end{align*}
\end{lemma}
\begin{proof}
    Consider a fixed $S \subseteq A$ of size $s \leq n_2/d_1$ and $W \subseteq [r]$ of size $w \geq \frac{r\log n}{d_1}$.
    Let $\ell \coloneqq \sum_{i\in W} |B_i|$.
    By symmetry, the distribution of $\bH$ is the same if we permute the vertices in $B$.
    Thus, the random variable $\sum_{i\in W}|N_{\bH}(S) \cap B_i|$ has the same distribution as $|N_{\bH}(S) \cap \bT|$ where $\bT \subseteq B$ is a uniformly random subset of size $\ell$, independent of $\bH$.

    Since $N_H(S) \leq d_1 s$ for any $H$ (with probability $1$), we can instead analyze the tail probabilities of $|U \cap \bT|$ for any fixed $U \subseteq B$ of size $d_1 s$.
    We have that $\Pr[|N_{\bH}(S) \cap \bT| \geq \lambda] \leq \Pr[|U \cap \bT| \geq \lambda]$.
    
    The random set $\bT$ can be viewed as $\ell$ samples from $B$ \emph{without} replacement.
    By \Cref{lem:without-replacement}, we get the same concentration bounds as if it is a sum of $\ell$ samples with replacement:
    letting $\mu \coloneqq \E[|U \cap \bT|] = d_1s\ell/n_2$, for any $\delta > 0$,
    \begin{align*}
        \Pr\bracks*{|U \cap \bT| \geq (1+\delta) \mu} \leq \exp\parens*{-\frac{\delta^2 \mu}{2+\delta}} \mper
    \end{align*}
    By assumption we have $w \cdot \frac{n_2}{2r} \leq \ell \leq w \cdot \frac{2n_2}{r}$.
    Thus, it follows that $d_1 s \frac{w}{2r} \leq \mu \leq d_1 s \frac{2w}{r}$.
    We now split into two cases depending on the size of $s$.
    Let the threshold be $\tau \coloneqq \frac{r\log n}{d_1}$, which corresponds to $\mu \approx w\log n$.
    Let $C = 8$.

    For $s \geq \tau$, we set $\delta = 2C-1$, and we have that $\Pr[|U \cap \bT| \geq 2C \mu] \leq \exp(-C\mu) \leq \exp(-Cd_1 s\frac{w}{2r})$.
    In this case, $2C\mu \leq 4Cd_1 s \frac{w}{r}$.

    For $s \leq \tau$, we set $\delta$ such that $(1+\delta) \mu = 2C w\log n$.
    In this case, we have $\frac{2C w\log n}{\mu} \geq \frac{Cr\log n}{d_1 s} \geq C \geq 5$, which means that $\delta \geq 4$ and hence $\frac{\delta^2}{2+\delta} \geq \frac{1}{2}(1+\delta)$.
    Thus, $\Pr[|U \cap \bT| \geq 2Cw\log n] \leq \exp(-Cw\log n)$.

    We next union bound over all $S\subseteq A$ of size $s$.
    \begin{align*}
        \sum_{s=1}^{\tau-1} \binom{n_1}{s} \cdot e^{-Cw\log n} + \sum_{s=\tau}^{n_2/d_1} \binom{n_1}{s} e^{-Cd_1s \frac{w}{2r}}
        \leq e^{\tau \log n_1 - Cw\log n} + \sum_{s=\tau}^{n_2/d_1} e^{s (\log n_1 - \frac{Cd_1w}{2r})} \mper
    \end{align*}
    Since we assume $w \geq \frac{r\log n}{d_1} = \tau$, for $C \geq 8$ we have $\frac{Cd_1w}{2r} \geq \frac{C}{2}\log n \geq 4\log n_1$.
    Thus, we can bound the above by $e^{-7w\log n} + e^{-\tau \cdot \frac{3d_1w}{r}} \leq e^{-2w\log n}$.

    Then, we union bound over $W \subseteq [r]$, which has at most $r^w \leq e^{w\log n}$ choices.
    Thus, with probability $1-O(1/n)$, we have
    \begin{equation*}
        \sum_{i\in W} |N_{\bH}(S) \cap B_i| \leq \max \braces*{ 4C d_1 s \frac{w}{r},\ 2Cw\log n}
        \leq 32 w \cdot \max \braces*{ \frac{d_1s}{r},\ \log n} \mper
        \qedhere
    \end{equation*}
\end{proof}

\section*{Acknowledgments}
S.M. would like to thank Omar Alrabiah, Louis Golowich, and Siqi Liu for insightful conversations.
J.H. would like to thank Mitali Bafna and Hung-Hsun Hans Yu for discussions on the Grassmann poset.

\bibliographystyle{alpha}
\bibliography{main}

\newcommand{\etalchar}[1]{$^{#1}$}
\begin{thebibliography}{HMMP24}

\bibitem[AC02]{AC02}
Noga Alon and Michael Capalbo.
\newblock Explicit unique-neighbor expanders.
\newblock In {\em The 43rd Annual IEEE Symposium on Foundations of Computer Science, 2002. Proceedings.}, pages 73--79. IEEE, 2002.

\bibitem[AD23]{AD23}
Ron Asherov and Irit Dinur.
\newblock {Bipartite unique-neighbour expanders via Ramanujan graphs}.
\newblock {\em arXiv preprint arXiv:2301.03072}, 2023.

\bibitem[ALOV19]{ALOV19}
Nima Anari, Kuikui Liu, Shayan Oveis{ }Gharan, and Cynthia Vinzant.
\newblock Log-concave polynomials {II}: high-dimensional walks and an {FPRAS} for counting bases of a matroid.
\newblock In {\em Proceedings of the 51st annual Symposium on Theory of Computing (STOC)}, pages 1--12, 2019.

\bibitem[Bal00]{Bal00}
Cristina Ballantine.
\newblock Ramanujan type buildings.
\newblock {\em Canadian Journal of Mathematics}, 52(6):1121--1148, 2000.

\bibitem[CGRZ24]{CGRZ2024}
Eshan Chattopadhyay, Mohit Gurumukhani, Noam Ringach, and Yunya Zhao.
\newblock Two-sided lossless expanders in the unbalanced setting, 2024.

\bibitem[Che24]{Chen2024}
Yeyuan Chen.
\newblock Unique-neighbor expanders with better expansion for polynomial-sized sets, 2024.

\bibitem[CRTS23]{CohenRTS23}
Itay Cohen, Roy Roth, and Amnon Ta-Shma.
\newblock Hdx condensers.
\newblock In {\em 2023 IEEE 64th Annual Symposium on Foundations of Computer Science (FOCS)}, pages 1649--1664, 2023.

\bibitem[CRVW02]{CRVW02}
Michael Capalbo, Omer Reingold, Salil Vadhan, and Avi Wigderson.
\newblock Randomness conductors and constant-degree lossless expanders.
\newblock In {\em Proceedings of the 34th Annual ACM Symposium on Theory of Computing}, pages 659--668, 2002.

\bibitem[CSZ03]{CSZ03}
Donald Cartwright, Patrick Sol\'{e}, and Andrzej \.{Z}uk.
\newblock Ramanujan geometries of type {$\widetilde{A}_n$}.
\newblock {\em Discrete Mathematics}, 269(1-3):35--43, 2003.

\bibitem[DH24]{DH24}
Yotam Dikstein and Max Hopkins.
\newblock {Chernoff Bounds and Reverse Hypercontractivity on HDX}.
\newblock {\em arXiv preprint arXiv:2404.10961}, 2024.

\bibitem[GHK{\etalchar{+}}22]{GHKLZ22}
Jason Gaitonde, Max Hopkins, Tali Kaufman, Shachar Lovett, and Ruizhe Zhang.
\newblock {Eigenstripping, Spectral Decay, and Edge-Expansion on Posets}.
\newblock In {\em Approximation, Randomization, and Combinatorial Optimization. Algorithms and Techniques (APPROX/RANDOM 2022)}. Schloss Dagstuhl-Leibniz-Zentrum f{\"u}r Informatik, 2022.

\bibitem[Gol23]{Gol23}
Louis Golowich.
\newblock {New Explicit Constant-Degree Lossless Expanders}.
\newblock {\em arXiv preprint arXiv:2306.07551}, 2023.

\bibitem[GUV09]{GUV09}
Venkatesan Guruswami, Christopher Umans, and Salil Vadhan.
\newblock {Unbalanced expanders and randomness extractors from Parvaresh--Vardy codes}.
\newblock {\em Journal of the ACM (JACM)}, 56(4):1--34, 2009.

\bibitem[HL22]{HL22}
Max Hopkins and Ting-Chun Lin.
\newblock Explicit lower bounds against $\omega$ (n)-rounds of sum-of-squares.
\newblock In {\em 2022 IEEE 63rd Annual Symposium on Foundations of Computer Science (FOCS)}, pages 662--673. IEEE, 2022.

\bibitem[HLW06]{HLW06}
Shlomo Hoory, Nathan Linial, and Avi Wigderson.
\newblock Expander graphs and their applications.
\newblock {\em Bulletin (new series) of the American Mathematical Society}, 43(4):439--561, 2006.

\bibitem[HMMP24]{HMMP24}
Jun-Ting Hsieh, Theo McKenzie, Sidhanth Mohanty, and Pedro Paredes.
\newblock Explicit two-sided unique-neighbor expanders.
\newblock In {\em Proceedings of the 56th Annual ACM Symposium on Theory of Computing}, pages 788--799, 2024.

\bibitem[Kah95]{Kah95}
Nabil Kahale.
\newblock Eigenvalues and expansion of regular graphs.
\newblock {\em Journal of the ACM (JACM)}, 42(5):1091--1106, 1995.

\bibitem[KK22]{KK22}
Amitay Kamber and Tali Kaufman.
\newblock {Combinatorics via closed orbits: number theoretic Ramanujan graphs are not unique neighbor expanders}.
\newblock In {\em Proceedings of the 54th Annual ACM SIGACT Symposium on Theory of Computing}, pages 426--435, 2022.

\bibitem[KRZS24]{KoppartyRZS24}
Swastik Kopparty, Noga Ron-Zewi, and Shubhangi Saraf.
\newblock Simple constructions of unique neighbor expanders from error-correcting codes, 2024.

\bibitem[KSY14]{KoppartySY14}
Swastik Kopparty, Shubhangi Saraf, and Sergey Yekhanin.
\newblock High-rate codes with sublinear-time decoding.
\newblock {\em J. ACM}, 61(5), September 2014.

\bibitem[KTS22]{KalevT22}
Itay Kalev and Amnon Ta-Shma.
\newblock {Unbalanced Expanders from Multiplicity Codes}.
\newblock In Amit Chakrabarti and Chaitanya Swamy, editors, {\em Approximation, Randomization, and Combinatorial Optimization. Algorithms and Techniques (APPROX/RANDOM 2022)}, volume 245 of {\em Leibniz International Proceedings in Informatics (LIPIcs)}, pages 12:1--12:14, Dagstuhl, Germany, 2022. Schloss Dagstuhl -- Leibniz-Zentrum f{\"u}r Informatik.

\bibitem[KY24]{KY24}
Dmitriy Kunisky and Xifan Yu.
\newblock Computational hardness of detecting graph lifts and certifying lift-monotone properties of random regular graphs.
\newblock {\em arXiv preprint arXiv:2404.17012}, 2024.

\bibitem[LH22]{LH22b}
Ting-Chun Lin and Min-Hsiu Hsieh.
\newblock {Good quantum LDPC codes with linear time decoder from lossless expanders}.
\newblock {\em arXiv preprint arXiv:2203.03581}, 2022.

\bibitem[Li04]{Li04}
Wen-Ching{ }Winnie Li.
\newblock Ramanujan hypergraphs.
\newblock {\em Geometric and Functional Analysis}, 14(2):380--399, 2004.

\bibitem[LSV05a]{LSV05b}
Alexander Lubotzky, Beth Samuels, and Uzi Vishne.
\newblock Explicit constructions of {R}amanujan complexes of type {$\tilde{A}_d$}.
\newblock {\em European Journal of Combinatorics}, 26(6):965--993, 2005.

\bibitem[LSV05b]{LSV05}
Alexander Lubotzky, Beth Samuels, and Uzi Vishne.
\newblock Ramanujan complexes of type {$\tilde{A}_d$}.
\newblock {\em Israel Journal of Mathematics}, 149:267--299, 2005.
\newblock Probability in mathematics.

\bibitem[Mor94]{Mor94}
Moshe Morgenstern.
\newblock Existence and explicit constructions of {$q+1$} regular {R}amanujan graphs for every prime power {$q$}.
\newblock {\em J. Combin. Theory Ser. B}, 62(1):44--62, 1994.

\bibitem[RVW00]{RVW00}
Omer Reingold, Salil Vadhan, and Avi Wigderson.
\newblock Entropy waves, the zig-zag graph product, and new constant-degree expanders and extractors.
\newblock In {\em Proceedings 41st Annual Symposium on Foundations of Computer Science}, pages 3--13. IEEE, 2000.

\bibitem[Sar04]{Sar04}
Alireza Sarveniazi.
\newblock {\em Ramanujan Hypergraph Based on {B}ruhat--{T}its Building}.
\newblock PhD thesis, University of G{\"o}ttingen, 2004.

\bibitem[SS96]{SS96}
Michael Sipser and Daniel Spielman.
\newblock Expander codes.
\newblock {\em IEEE Trans. Inform. Theory}, 42(6, part 1):1710--1722, 1996.

\bibitem[TSUZ01]{TaShmaUZ01}
Amnon Ta-Shma, Christopher Umans, and David Zuckerman.
\newblock Loss-less condensers, unbalanced expanders, and extractors.
\newblock In {\em Proceedings of the Thirty-Third Annual ACM Symposium on Theory of Computing}, STOC '01, page 143–152, New York, NY, USA, 2001. Association for Computing Machinery.

\end{thebibliography}

\end{document}